\documentclass[11pt]{article}

\usepackage[a4paper]{geometry}
\usepackage{graphicx,amssymb,mathrsfs,amsmath}
\usepackage{float,amsthm,subfigure,hyperref} 
\usepackage{algorithm,algorithmic,indentfirst}
\usepackage{booktabs,paralist,lettrine}
\usepackage{soul,color,xcolor,arydshln}

\newtheorem{theorem}{Theorem}[section]
\newtheorem{definition}{Definition}[section]

\newtheorem{lemma}{Lemma}[section]
\newtheorem{corollary}{Corollary}[section]

\newtheorem{remark}{Remark}[section]
\newtheorem{proposition}{Proposition}[section]
\numberwithin{equation}{section}

\begin{document}

\title{On Rayleigh Quotient Iteration for Dual Quaternion Hermitian Eigenvalue Problem\thanks{This work is supported by the National Natural Science Foundation of China under Grants 12371023, 12201149, 12361079, and the Natural Science Foundation of Guangxi Province  under Grant 2023GXNSFAA026067.}
}
\author
{Shan-Qi Duan\thanks{Department of Mathematics and Newtouch Center for Mathematics, Shanghai University, Shanghai 200444, P.R. China (Email: {\tt shanqiduan1020@126.com})}
\quad Qing-Wen Wang\thanks{Corresponding author. Department of Mathematics and Newtouch Center for Mathematics; Collaborative Innovation Center for the Marine Artificial Intelligence, Shanghai University, Shanghai 200444, P.R. China (Email: {\tt wqw@t.shu.edu.cn})}
\quad Xue-Feng Duan\thanks{College of Mathematics and Computational Science, Guilin University of Electronic Technology, Guilin 541004, P.R. China (Email: {\tt guidian520@126.com})}
}

\date{}
\maketitle

\begin{abstract}

The application of eigenvalue theory to dual quaternion Hermitian matrices holds significance in the realm of multi-agent formation control. In this paper, we study the Rayleigh quotient iteration (RQI) for solving the right eigenpairs of dual quaternion Hermitian matrices. Combined with dual representation, the RQI algorithm can effectively compute the eigenvalue along with the associated eigenvector of the dual quaternion Hermitian matrices. Furthermore, by utilizing minimal residual property of the Rayleigh Quotient, a convergence analysis of the Rayleigh quotient iteration is derived. Numerical examples are provided to illustrate the high accuracy and low CPU time cost of the proposed Rayleigh quotient iteration compared with the power method for solving the dual quaternion Hermitian eigenvalue problem.


\end{abstract}

\section{Introduction}

\noindent This paper concerns the dual quaternion Hermitian right eigenvalue problem
\begin{equation}\label{problem}
\hat{A} \hat{\mathbf{x}}=\hat{\mathbf{x}} \hat{\lambda},
\end{equation}
where $\hat{A} \in \mathbb{DQ}^{n \times n}$ is the objective dual quaternion Hermitian matrix, $\hat{\lambda} \in \mathbb{DQ}$ and $\hat{\mathbf{x}} \in \mathbb{DQ}^{n}$ constitute the right eigenpair of $\hat{A}$.

Dual quaternion numbers and dual quaternion matrices have been widely applied in rigid body motions \cite{brambley2020unit,kenright2012biginners,cheng2016dual,wang2012dual}, hand-eye calibration problems \cite{daniilidis1999hand,chen2024dual}, simultaneous localization and mapping (SLAM) problems \cite{bryson2007building,cadena2016past,carlone2015initialization}, and kinematic modeling and control \cite{qi2023dualcontrol}. Qi et al. \cite{qi2022dual} first introduced a total order and an absolute value function for dual numbers. In addition, they also defined the magnitude of a dual quaternion and extended the norm to dual quaternion vectors. In \cite{qi2021eigenvalues}, Qi and Luo investigated the right and left eigenvalues of square dual quaternion matrices. If a right eigenvalue corresponds to a dual number, it is also a left eigenvalue. In such cases, this dual number is referred to as an eigenvalue of the dual quaternion matrix. The researchers demonstrated that the right eigenvalues of a dual quaternion Hermitian matrix are themselves dual numbers, making them valid eigenvalues. Furthermore, it was established that an $n$-by-$n$ dual quaternion Hermitian matrix possesses precisely $n$ eigenvalues. The unitary decomposition for a dual quaternion Hermitian matrix and singular value decomposition for a general dual quaternion matrix were also proposed. According to the singular value decomposition for a general dual quaternion matrix, Ling et al. \cite{ling2022singular} defined the rank and appreciable rank of dual quaternion matrices, and some properties were studied. In \cite{qi2023dualcontrol}, the significance of the eigenvalue theory of dual quaternion Hermitian matrices in multi-agent formation control was demonstrated. 

In \cite{cui2023power}, Cui and Qi first
proposed a power method for efficiently computing the dominant eigenvalue of a dual quaternion Hermitian matrix and reformulated the simultaneous localization and mapping (SLAM) problem as a rank-one dual quaternion completion problem, leveraging the findings from the eigenvalue analysis. So far, the methods for numerically computing the eigenvalues of a dual quaternion Hermitian matrix are still rarely explored.

There are several methods to solve the right quaternion eigenvalue problem, such as the power method \cite{li2019power}, and inverse iteration \cite{jia2023computing}. When aiming to compute eigenvalues of a Hermitian matrix around a specified target, a common approach involves using a shift-and-invert transformation. This leads to effective iteration methods like the inverse iteration and the Rayleigh quotient iteration (RQI) \cite{bai2019multistep,parlett1998symmetric}. The RQI, in particular, utilizes a shift based on the Rayleigh quotient for the currently available approximate eigenvector. As a result, the RQI exhibits a nonstationary iteration process with a local cubic convergence rate \cite{berns2006inexact,notay2003convergence,parlett1974rayleigh,saad2011numerical,smit1999effects}. Since RQI has a faster convergence rate, we develop the Rayleigh quotient iteration to compute the eigenpairs of the dual quaternion Hermitian matrix in this paper.

In this paper, we study the problem of solving the right eigenvalues for dual quaternion Hermitian matrices. The main contributions of this paper include the following: Firstly, we propose a Rayleigh quotient iteration (RQI) method for computing the eigenvalue and the associated eigenvector of the dual quaternion Hermitian matrix. Furthermore, by utilizing minimal residual property of the Rayleigh quotient, a convergence analysis of the Rayleigh quotient iteration is derived. Secondly, we demonstrate the superior accuracy and efficiency of the proposed RQI method for solving dual quaternion Hermitian eigenvalue problem, e.g., the residual errors are several orders of magnitudes smaller than the power method. Thirdly, we develop the dual representation of the dual quaternion matrix. Specifically, we establish an isomorphic mapping between the dual quaternion matrix and the dual matrix. This dual representation plays a crucial role in solving dual quaternion linear systems. The proposed dual representation also lays a foundation for the development of post-renewal structure-preserving algorithms.

The rest of this paper is organized as follows. In Section \ref{section2}, we introduce some notations used throughout this paper. In addition, we review some basic properties of the dual quaternion matrices. In Section \ref{section3}, the Rayleigh quotient iteration (RQI) for computing the eigenpair is proposed. In Section \ref{section4}, we derive a convergence analysis of the Rayleigh quotient iteration. In Section \ref{section5}, numerical experiments on dual quaternion Hermitian eigenvalue problems are provided to illustrate the effectiveness of the proposed method compared with the power method. In Section \ref{section6}, we summarize this paper.

\section{Dual Quaternions and Dual Quaternion Matrices }\label{section2}

\noindent In this section, some notations and basic definitions are introduced, which will be used in the rest of the paper.

\subsection{Dual numbers, quaternions and dual quaternions}

\noindent Dual number was invented by Clifford \cite{clifford1871preliminary} in 1873. Denote the set of dual numbers as
\begin{equation*}
\mathbb{D}=\left\{q=q_{\mathrm{st}}+q_{\mathcal{I}}\epsilon  ~ | ~ q_{\mathrm{st}},~ q_{\mathcal{I}}\in\mathbb{R} \right\},
\end{equation*}
where $\epsilon$ is the infinitesimal unit, satisfying $\epsilon^2=0$. We call $q_{\mathrm{st}}$ the real part or the standard part of $q$, and $q_{\mathcal{I}}$ the dual part or the infinitesimal part of $q$. The infinitesimal unit $\epsilon$ is commutative in multiplication with real numbers, complex numbers and quaternions. Dual numbers constitute a commutative algebra with a dimension of two over the real number field. When $q_{\mathrm{st}}\neq 0$, we say that $q$ is appreciable; otherwise, it is deemed infinitesimal.

The infinitesimal unit $\epsilon$ in dual numbers is an algebraic construct representing an infinitesimally small quantity. In fact, $\epsilon$ is a formal symbol representing an infinitesimal change. It's not a real number but a mathematical tool that allows us to capture the idea of an infinitely small quantity. In geometry, the infinitesimal element $\epsilon$ can be seen as representing an infinitesimal displacement or rotation. This makes dual numbers useful for modeling small transformations in kinematics. The dual quaternion comprises two parts that can describe the rigid-body transformation of one coordinate frame with respect to another, where the standard part and dual part are quaternions representing the rotation and translation, respectively \cite{brambley2020unit,cheng2016dual,kenright2012biginners}.

In \cite{qi2022dual}, a total order for dual numbers was introduced. Given two dual numbers $p=p_{\mathrm{st}}+p_{\mathcal{I}}\epsilon,q=q_{\mathrm{st}}+q_{\mathcal{I}}\epsilon \in \mathbb{D}$, we have $p < q$, if either $p_{\mathrm{st}} <q_{\mathrm{st}}$ or $p_{\mathrm{st}}=q_{\mathrm{st}}$ and $p_{\mathcal{I}} < q_{\mathcal{I}}$; If $p_{\mathrm{st}}=q_{\mathrm{st}}$ and $p_{\mathcal{I}}= q_{\mathcal{I}}$, then $p = q$. Specifically, we call $p$ positive, nonnegative, nonpositive, or negative when it satisfies the conditions $p > 0,~ p \geq 0,~ p \leq 0$ or $p < 0$, respectively.

The zero element in $\mathbb{D}$ is $0_{D}=0+0 \epsilon$ and the unit element is $1_{D}=1+0 \epsilon$. For two dual numbers $p=p_{\mathrm{st}}+p_{\mathcal{I}}\epsilon,~ q=q_{\mathrm{st}}+q_{\mathcal{I}}\epsilon \in \mathbb{D}$, addition and multiplication of $p$ and $q$ are defined as follows
\begin{equation*}
\begin{aligned}
	& p+q= p_{\mathrm{st}}+q_{\mathrm{st}}+(p_{\mathcal{I}}+q_{\mathcal{I}})\epsilon,\\
	& pq=p_{\mathrm{st}}q_{\mathrm{st}}+(p_{\mathrm{st}}q_{\mathcal{I}}+p_{\mathcal{I}}q_{\mathrm{st}})\epsilon.
\end{aligned}
\end{equation*}
And the division of $p$ and $q$ is unambiguous if $q_{\mathrm{st}}\neq 0$ is appreciable and takes the form \cite{VELDKAMP1976141}
\begin{equation*}
\frac{p_{\mathrm{st}}+p_{\mathcal{I}} \epsilon}{q_{\mathrm{st}}+q_{\mathcal{I}} \epsilon}=
\frac{p_{\mathrm{st}}}{q_{\mathrm{st}}}+\left(\frac{p_{\mathcal{I}}}{q_{\mathrm{st}}}-\frac{p_{\mathrm{st}}q_{\mathcal{I}}}{q_{\mathrm{st}}^{2}} \right) \epsilon, ~ q_{\mathrm{st}} \neq 0.
\end{equation*}

For an appreciable dual number $q=q_{\mathrm{st}}+q_{\mathcal{I}}\epsilon$, it is invertible and
\begin{equation*}
q^{-1}=q_{\mathrm{st}}^{-1}-q_{\mathrm{st}}^{-1} q_{\mathcal{I}} q_{\mathrm{st}}^{-1} \epsilon,
\end{equation*}
as indicated by the equation $q q^{-1}=q^{-1} q=1$. If $q$ is infinitesimal, then $q$ is not invertible. For two dual numbers $p=p_{\mathrm{st}}+p_{\mathcal{I}}\epsilon,~ q=q_{\mathrm{st}}+q_{\mathcal{I}}\epsilon \in \mathbb{D}$, we have
\begin{equation*}
	\frac{p_{\mathrm{st}}+p_{\mathcal{I}} \epsilon}{q_{\mathrm{st}}+q_{\mathcal{I}} \epsilon}=
	\frac{p_{\mathrm{st}}}{q_{\mathrm{st}}}+\left(\frac{p_{\mathcal{I}}}{q_{\mathrm{st}}}-\frac{p_{\mathrm{st}}q_{\mathcal{I}}}{q_{\mathrm{st}}^{2}} \right) \epsilon=(p_{\mathrm{st}}+p_{\mathcal{I}} \epsilon)q^{-1}, \ q_{\mathrm{st}} \neq 0.
\end{equation*}One can observe that the division of dual numbers is the inverse operation of the multiplication of dual numbers.

For a nonnegative and appreciable dual number $q$, the square root of $q$ remains a nonnegative dual number. If $q$ is positive and appreciable, then
\begin{equation*}
\sqrt{q}=\sqrt{q_{\mathrm{st}}}+\frac{q_{\mathcal{I}}}{2 \sqrt{q_{\mathrm{st}}}} \epsilon.
\end{equation*}

The absolute value of $p=p_{\mathrm{st}}+p_{\mathcal{I}}\epsilon \in \mathbb{D}$ is defined as
$$
|p|=\begin{cases}
\left|p_{\mathrm{st}}\right|+\operatorname{sgn}\left(p_{\mathrm{st}}\right) p_{\mathcal{I}} \epsilon, & \mbox{if} ~ p_{\mathrm{st}} \neq 0, \\
\left|p_{\mathcal{I}}\right| \epsilon, & \mbox{otherwise},
\end{cases}
$$
where ``$\operatorname{sgn}(\cdot)$" is given by
\begin{equation*}
\operatorname{sgn}(a)=\begin{cases}
	1, & \mbox{if} ~ a>0, \\
	0, & \mbox{if} ~ a=0, \\
	-1, & \mbox{if} ~ a<0,
\end{cases} \quad \forall a \in \mathbb{R}.
\end{equation*}
For more details on dual numbers, see \cite{qi2022dual} and the references therein.

Denote the set of quaternions as
\begin{equation*}
\mathbb{Q}=\left\{\tilde{q}=q_{1}+q_{2}\mathbf{i}+q_{3}\mathbf{j}+q_{4}\mathbf{k}~|~q_{1},q_{2},q_{3},q_{4}\in \mathbb{R}\right\},
\end{equation*}
where $\mathbf{i}, \mathbf{j}, \mathbf{k}$ are three imaginary units of quaternions, satisfying 
\begin{equation*}
\begin{aligned}
	& \mathbf{i}^{2}=\mathbf{j}^{2}=\mathbf{k}^{2}=\mathbf{ijk}=-1, \\
	& \mathbf{ij}=-\mathbf{ji}=\mathbf{k}, \ \mathbf{jk}=-\mathbf{kj}=\mathbf{i}, \ \mathbf{ki}=-\mathbf{ik}=\mathbf{j}.
\end{aligned}
\end{equation*} 
The real part of $\tilde{q}$ is $\mathrm{Re}(\tilde{q}) = q_{1}$. The imaginary part of $\tilde{q}$ is $\mathrm{Im}(\tilde{q}) = q_{2}\mathbf{i}+q_{3}\mathbf{j}+q_{4}\mathbf{k}$. A quaternion is called imaginary when its real part is equal to zero. The multiplication of quaternions adheres to the distributive law but is noncommutative. 

The zero element in $\mathbb{Q}$ is $\tilde{0}=0+0\mathbf{i}+0\mathbf{j}+0\mathbf{k}$ and the unit element is $\tilde{1}=1+0\mathbf{i}+0\mathbf{j}+0\mathbf{k}$. For any $\tilde{q}=q_{1}+q_{2}\mathbf{i}+q_{3}\mathbf{j}+q_{4}\mathbf{k} \in \mathbb{Q}$, the conjugate of a quaternion is defined as 
\begin{equation*}
\tilde{q}^{\ast}=q_{1}-q_{2}\mathbf{i}-q_{3}\mathbf{j}-q_{4}\mathbf{k}.
\end{equation*}
The magnitude of $\tilde{q}$ is $|\tilde{q}|=\sqrt{\tilde{q}^{\ast}\tilde{q}}=\sqrt{q_{1}^{2}+q_{2}^{2}+q_{3}^{2}+q_{4}^{2}}$, it follows that the inverse of a nonzero quaternion $\tilde{q}$ is given by $\tilde{q}^{-1}=\tilde{q}^{*} / |\tilde{q}|^{2}$.

Denote the set of dual quaternions as 
\begin{equation*}
\mathbb{DQ}=\left\{\hat{q}=\tilde{q}_{\mathrm{st}}+\tilde{q}_{\mathcal{I}}\epsilon  ~ | ~ \tilde{q}_{\mathrm{st}},~ \tilde{q}_{\mathcal{I}}\in\mathbb{Q} \right\}.
\end{equation*}
Similarly, we call $\tilde{q}_{\mathrm{st}},\tilde{q}_{\mathcal{I}}$ the standard part and the dual part of $\hat{q}$, respectively. If $\tilde{q}_{\mathrm{st}} \neq 0$, then $\hat{q}$ is appreciable, otherwise, $\hat{q}$ is infinitesimal.
Let $\hat{p}=\tilde{p}_{\mathrm{st}}+\tilde{p}_{\mathcal{I}} \epsilon, \hat{q}=\tilde{q}_{\mathrm{st}}+\tilde{q}_{\mathcal{I}} \epsilon \in \mathbb{DQ}$, then the addition of $\hat{p}$ and $\hat{q}$ is
$$
\hat{p}+\hat{q}=\tilde{p}_{\mathrm{st}}+\tilde{q}_{\mathrm{st}}+(\tilde{p}_{\mathcal{I}}+\tilde{q}_{\mathcal{I}})\epsilon
$$
and the product of $\hat{p}$ and $\hat{q}$ is
\begin{equation*}
\hat{p}\hat{q}=\tilde{p}_{\mathrm{st}} \tilde{q}_{\mathrm{st}}+(\tilde{p}_{\mathrm{st}} \tilde{q}_{\mathcal{I}}+\tilde{p}_{\mathcal{I}} \tilde{q}_{\mathrm{st}})\epsilon.
\end{equation*}

The zero element in $\mathbb{DQ}$ is $\hat{0}=\tilde{0}+\tilde{0} \epsilon$ and the unit element is $\hat{1}=\tilde{1}+\tilde{0} \epsilon$. In general, $\hat{p}\hat{q} \neq \hat{q}\hat{p}$. The conjugate of  $\hat{p}=\tilde{p}_{\mathrm{st}}+\tilde{p}_{\mathcal{I}} \epsilon$ is  $\hat{p}^{\ast}=\tilde{p}_{\mathrm{st}}^{\ast}+\tilde{p}_{\mathcal{I}}^{\ast} \epsilon$. A dual quaternion $\hat{q}=\tilde{q}_{\mathrm{st}}+\tilde{q}_{\mathcal{I}} \epsilon$ is invertible if and only if $\hat{q}$ is appreciable. In this case, we have
$$
\hat{q}^{-1}=\tilde{q}_{\mathrm{st}}^{-1}-\tilde{q}_{\mathrm{st}}^{-1} \tilde{q}_{\mathcal{I}} \tilde{q}_{\mathrm{st}}^{-1} \epsilon.
$$
The magnitude of $\hat{q}$ is defined as
$$
|\hat{q}|=\begin{cases}
\left|\tilde{q}_{\mathrm{st}}\right|+\dfrac{\tilde{q}_{\mathrm{st}} \tilde{q}_{\mathcal{I}}^{\ast}+\tilde{q}_{\mathcal{I}} \tilde{q}_{\mathrm{st}}^{\ast}}{2\left|\tilde{q}_{\mathrm{st}}\right|} \epsilon, & \mbox{if} ~ \tilde{q}_{\mathrm{st}} \neq 0, \\
\left|\tilde{q}_{\mathcal{I}}\right| \epsilon, & \mbox{otherwise},
\end{cases}
$$
which is a dual number. A dual quaternion  $\hat{q}=\tilde{q}_{\mathrm{st}}+\tilde{q}_{\mathcal{I}} \epsilon$ is called a unit dual quaternion if
$$
\left|\tilde{q}_{\mathrm{st}}\right|=1 ~ \text{and} ~ \tilde{q}_{\mathrm{st}} \tilde{q}_{\mathcal{I}}^{\ast}+\tilde{q}_{\mathcal{I}} \tilde{q}_{\mathrm{st}}^{\ast}=0.$$
Further, for any dual quaternion number  $\hat{q}=\tilde{q}_{\mathrm{st}}+\tilde{q}_{\mathcal{I}} \epsilon \in \mathbb{DQ}$ and dual number  $a=a_{\mathrm{st}}+a_{\mathcal{I}} \epsilon \in \mathbb{D}$, there is
$$
\frac{\hat{q}}{a}=\frac{\tilde{q}_{\mathrm{st}}+\tilde{q}_{\mathcal{I}} \epsilon}{a_{\mathrm{st}}+a_{\mathcal{I}} \epsilon}=
\frac{\tilde{q}_{\mathrm{st}}}{a_{\mathrm{st}}}+\left(\frac{\tilde{q}_{\mathcal{I}}}{a_{\mathrm{st}}}-\frac{\tilde{q}_{\mathrm{st}}}{a_{\mathrm{st}}} \frac{a_{\mathcal{I}}}{a_{\mathrm{st}}}\right) \epsilon, ~ a_{\mathrm{st}} \neq 0.
$$

\subsection{Dual quaternion matrices}

\noindent Denote the set of dual quaternion matrices as
\begin{equation*}
\mathbb{DQ}^{m\times n}=\left\{\hat{A}=\tilde{A}_{\mathrm{st}}+\tilde{A}_{\mathcal{I}}\epsilon  ~ | ~ \tilde{A}_{\mathrm{st}}, ~ \tilde{A}_{\mathcal{I}}\in\mathbb{Q}^{m\times n} \right\}.
\end{equation*}
Again, we call $\tilde{A}_{\mathrm{st}}, \tilde{A}_{\mathcal{I}}$ the standard part and the dual part of $\hat{A}$, respectively. The conjugate transpose of $\hat{A}$ is $\hat{A}^{\ast}=\tilde{A}_{\mathrm{st}}^{\ast}+\tilde{A}_{\mathcal{I}}^{\ast} \epsilon$.

The zero matrix in $\mathbb{DQ}^{m \times n}$ is $\hat{\mathbf{O}}=\tilde{\mathbf{O}}+\tilde{\mathbf{O}} \epsilon$ and the identity matrix is $\hat{\mathbf{I}}=\tilde{\mathbf{I}}+\tilde{\mathbf{O}} \epsilon$. We say that a square dual quaternion matrix $\hat{A} \in \mathbb{DQ}^{n \times n}$ is normal if $\hat{A}^{\ast} \hat{A}=\hat{A} \hat{A}^{\ast}$; Hermitian if $\hat{A}^{\ast}=\hat{A}$; Unitary if $\hat{A}^{\ast} \hat{A}=I$, where $I$ is the identity matrix; Invertible (nonsingular) if there exists a matrix $\hat{B} \in \mathbb{DQ}^{n \times n}$ such that $\hat{A}\hat{B}=\hat{B}\hat{A}=I$. In this case, we denote $\hat{A}^{-1}=\hat{B}$. We have $(\hat{A}\hat{B})^{-1}=\hat{B}^{-1}\hat{A}^{-1}$ if $\hat{A}$ and $\hat{B}$ are invertible, and  $(\hat{A}^{\ast})^{-1}=(\hat{A}^{-1})^{\ast}$ if $\hat{A}$ is invertible.

Given a dual quaternion Hermitian matrix $\hat{A} \in \mathbb{DQ}^{n \times n}$, for any  $\hat{\mathbf{x}} \in \mathbb{DQ}^{n}$, the following equality holds:
$$
(\hat{\mathbf{x}}^{*} \hat{A} \hat{\mathbf{x}})^{\ast}=\hat{\mathbf{x}}^{\ast} \hat{A} \hat{\mathbf{x}}.
$$
This implies that $\hat{\mathbf{x}}^{\ast} \hat{A} \hat{\mathbf{x}}$ is a dual number. A dual quaternion Hermitian matrix $\hat{A} \in \mathbb{DQ}^{n \times n}$ is called positive semidefinite if for any $\hat{\mathbf{x}} \in \mathbb{DQ}^{n},  \hat{\mathbf{x}}^{\ast} \hat{A} \hat{\mathbf{x}} \geq \hat{0}$; $\hat{A}$ is called positive definite if for any  $\hat{\mathbf{x}} \in \mathbb{DQ}^{n}$ with  $\hat{\mathbf{x}}$ being appreciable, we have  $\hat{\mathbf{x}}^{\ast} \hat{A} \hat{\mathbf{x}} > \hat{0}$ and is appreciable. A square quaternion matrix $\hat{A} \in \mathbb{DQ}^{n \times n}$ is unitary if and only if its column (row) vectors constitute an orthogonal basis of $\mathbb{DQ}^{n}$.

The $2$-norm of a dual quaternion vector $\hat{\mathbf{x}}=(\hat{x}_{1},\cdots,\hat{x}_{n})^{\top}\in \mathbb{DQ}^{n \times 1}$ is
\begin{equation*}
\|\hat{\mathbf{x}}\|_{2}=\begin{cases}
	\sqrt{\sum_{i=1}^{n}\left|\hat{x}_{i}\right|^{2}}, & \mbox{if} ~ \tilde{\mathbf{x}}_{\mathrm{st}} \neq \tilde{\mathbf{0}}, \\
	\sqrt{\sum_{i=1}^{n}\left|\tilde{x}_{i,\mathcal{I}}\right|^{2}} \epsilon, & \mbox {if} ~ \tilde{\mathbf{x}}_{\mathrm{st}}=\tilde{\mathbf{0}} ~ \text{and} ~ \hat{\mathbf{x}}=\tilde{\mathbf{x}}_{\mathcal{I}} \epsilon.
\end{cases}
\end{equation*}

For two given dual quaternion vectors $\hat{\mathbf{x}}=\left(\hat{x}_{1}, \hat{x}_{2}, \ldots, \hat{x}_{n}\right)^{\top}, \hat{\mathbf{y}}=\left(\hat{y}_{1}, \hat{y}_{2}, \ldots, \hat{y}_{n}\right)^{\top}\in \mathbb{DQ}^{n}$, the dual quaternion-valued inner product $\langle\hat{\mathbf{x}}, \hat{\mathbf{y}}\rangle$ is defined as $$
\langle\hat{\mathbf{x}}, \hat{\mathbf{y}}\rangle=\hat{\mathbf{y}}^{\ast}\hat{\mathbf{x}}=\sum_{i=1}^{n} \hat{y}_{i}^{*} \hat{x}_{i}.$$ 
It is easy to see that $\langle\hat{\mathbf{x}} \hat{a}+\hat{\mathbf{z}} \hat{b}, \hat{\mathbf{y}}\rangle=\langle\hat{\mathbf{x}}, \hat{\mathbf{y}}\rangle \hat{a}+\langle\hat{\mathbf{z}}, \hat{\mathbf{y}}\rangle \hat{b}$ and  $\langle\hat{\mathbf{x}}, \hat{\mathbf{y}}\rangle=\langle\hat{\mathbf{y}}, \hat{\mathbf{x}}\rangle^{\ast}$ for any $\hat{\mathbf{x}}, \hat{\mathbf{y}}, \hat{\mathbf{z}} \in \mathbb{DQ}^{n}$ and  $\hat{a}, \hat{b} \in \mathbb{DQ}$. We say that $\hat{\mathbf{x}}$ and  $\hat{\mathbf{y}}$ are orthogonal to each other if  $\langle\hat{\mathbf{x}}, \hat{\mathbf{y}}\rangle=0$, denote by $\hat{\mathbf{x}} \perp \hat{\mathbf{y}}$. And for any $\hat{\mathbf{x}} \in \mathbb{DQ}^{n}$, $\langle\hat{\mathbf{x}}, \hat{\mathbf{x}}\rangle$ is a nonnegative dual number, and if $\hat{\mathbf{x}}$ is appreciable, $\langle\hat{\mathbf{x}}, \hat{\mathbf{x}}\rangle$ is a positive dual number \cite{qi2022dual}.

\begin{lemma}\label{induce}
	Given an appreciable dual quaternion vector $\hat{\mathbf{x}}\in \mathbb{DQ}^{n}$, the norm $\|\cdot\|_{2}$ is induced by the inner product, i.e. $\|\hat{\mathbf{x}}\|_{2}=\sqrt{\langle\hat{\mathbf{x}},\hat{\mathbf{x}}\rangle}$.
\end{lemma}

\begin{proof}
	Write $\hat{\mathbf{x}}=\tilde{\mathbf{x}}_{\mathrm{st}}+\tilde{\mathbf{x}}_{\mathcal{I}}\epsilon=(\hat{x}_{1},\cdots,\hat{x}_{n})^{\top}\in \mathbb{DQ}^{n}$. The dual quaternion-valued inner product $\langle\hat{\mathbf{x}}, \hat{\mathbf{x}}\rangle$ is defined by \begin{equation*}
		\langle\hat{\mathbf{x}}, \hat{\mathbf{x}}\rangle=\hat{\mathbf{x}}^{\ast}\hat{\mathbf{x}}=\sum_{i=1}^{n} \hat{x}_{i}^{*} \hat{x}_{i},
	\end{equation*}
	and the 2-norm of $\hat{\mathbf{x}}$ is defined by
	\begin{equation*}
		\|\hat{\mathbf{x}}\|_{2}=\sqrt{\sum_{i=1}^{n}\left|\hat{x}_{i}\right|^{2}}, ~ \text{if} ~ \tilde{\mathbf{x}}_{\mathrm{st}} \neq \tilde{\mathbf{0}}.
	\end{equation*}
	Let $\hat{x}_{i}=\tilde{x}_{i,\mathrm{st}}+\tilde{x}_{i,\mathcal{I}}\epsilon\in \mathbb{DQ}$, $\forall i=1,2,\cdots,n$. Then it is easily verified that
	\begin{equation*}
		\hat{x}_{i}^{*} \hat{x}_{i}=(\tilde{x}_{i,\mathrm{st}}^{\ast}+\tilde{x}_{i,\mathcal{I}}^{\ast}\epsilon)(\tilde{x}_{i,\mathrm{st}}+\tilde{x}_{i,\mathcal{I}}\epsilon)=\tilde{x}_{i,\mathrm{st}}^{\ast}\tilde{x}_{i,\mathrm{st}}+(\tilde{x}_{i,\mathrm{st}}^{\ast}\tilde{x}_{i,\mathcal{I}}+\tilde{x}_{i,\mathcal{I}}^{\ast}\tilde{x}_{i,\mathrm{st}})\epsilon,
	\end{equation*}
	and
	\begin{equation*}
		|\hat{x}_{i}|^{2}=\left(\left|\tilde{x}_{i,\mathrm{st}}\right|+\dfrac{\tilde{x}_{i,\mathrm{st}} \tilde{x}_{i,\mathcal{I}}^{\ast}+\tilde{x}_{i,\mathcal{I}} \tilde{x}_{i,\mathrm{st}}^{\ast}}{2\left|\tilde{x}_{i,\mathrm{st}}\right|} \epsilon\right)^{2}=|\tilde{x}_{i,\mathrm{st}}|^{2}+(\tilde{x}_{i,\mathrm{st}}^{\ast}\tilde{x}_{i,\mathcal{I}}+\tilde{x}_{i,\mathcal{I}}^{\ast}\tilde{x}_{i,\mathrm{st}})\epsilon=\hat{x}_{i}^{*} \hat{x}_{i}.
	\end{equation*}
	which implies that $\hat{x}_{i}^{*} \hat{x}_{i}=|\hat{x}_{i}|^{2}$. This complete the proof.
\end{proof}

\begin{lemma}\label{normsum}
	Given two dual quaternion vectors $\mathbf{x,y}\in \mathbb{DQ}^{n}$, the following properties hold:
	
	\noindent $(1)$ $\|\hat{\mathbf{x}}\|_{2}=\|\hat{\mathbf{x}}^{\ast}\|_{2}$;
	
	\noindent $(2)$ $\|\hat{\mathbf{x}}+\hat{\mathbf{y}}\|_{2}^{2}=\|\hat{\mathbf{x}}\|_{2}^{2}+\|\hat{\mathbf{y}}\|_{2}^{2}+\langle\hat{\mathbf{x}},\hat{\mathbf{y}}\rangle+\langle\hat{\mathbf{y}},\hat{\mathbf{x}}\rangle$
\end{lemma}

\begin{proof}
	The property (1) is obvious and it is omitted here. Let
	\begin{equation*}
	\hat{\mathbf{x}}=(\hat{x}_{1},\cdots,\hat{x}_{n})^{\top}, \hat{\mathbf{y}}=(\hat{y}_{1},\cdots,\hat{y}_{n})^{\top}\in \mathbb{DQ}^{n \times 1},
	\end{equation*} 
	where
	$\hat{x}_{i}=\tilde{x}_{i,\mathrm{st}}+\tilde{x}_{i,\mathcal{I}}\epsilon$, 	$\hat{y}_{i}=\tilde{y}_{i,\mathrm{st}}+\tilde{y}_{i,\mathcal{I}}\epsilon \in \mathbb{DQ}$, then we have
	\begin{equation*}
		\begin{aligned}
			|\hat{x}_{i}+\hat{y}_{i}|^{2}&=\left(\left|\tilde{x}_{i,\mathrm{st}}+\tilde{y}_{i,\mathrm{st}}\right|+\dfrac{(\tilde{x}_{i,\mathrm{st}}+\tilde{y}_{i,\mathrm{st}}) (\tilde{x}_{i,\mathcal{I}}^{\ast}+\tilde{y}_{i,\mathcal{I}}^{\ast})+(\tilde{x}_{i,\mathcal{I}}+\tilde{y}_{i,\mathcal{I}}) (\tilde{x}_{i,\mathrm{st}}^{\ast}+\tilde{y}_{i,\mathrm{st}}^{\ast})}{2\left|\tilde{x}_{i,\mathrm{st}}+\tilde{y}_{i,\mathrm{st}}\right|} \epsilon\right)^{2}\\
			&=|\tilde{x}_{i,\mathrm{st}}+\tilde{y}_{i,\mathrm{st}}|^{2}+((\tilde{x}_{i,\mathrm{st}}+\tilde{y}_{i,\mathrm{st}}) (\tilde{x}_{i,\mathcal{I}}^{\ast}+\tilde{y}_{i,\mathcal{I}}^{\ast})+(\tilde{x}_{i,\mathcal{I}}+\tilde{y}_{i,\mathcal{I}}) (\tilde{x}_{i,\mathrm{st}}^{\ast}+\tilde{y}_{i,\mathrm{st}}^{\ast}))\epsilon \\
			&=(\tilde{x}_{i,\mathrm{st}}^{\ast}+\tilde{y}_{i,\mathrm{st}}^{\ast})(\tilde{x}_{i,\mathrm{st}}+\tilde{y}_{i,\mathrm{st}})+((\tilde{x}_{i,\mathrm{st}}+\tilde{y}_{i,\mathrm{st}}) (\tilde{x}_{i,\mathcal{I}}^{\ast}+\tilde{y}_{i,\mathcal{I}}^{\ast})+(\tilde{x}_{i,\mathcal{I}}+\tilde{y}_{i,\mathcal{I}}) (\tilde{x}_{i,\mathrm{st}}^{\ast}+\tilde{y}_{i,\mathrm{st}}^{\ast}))\epsilon \\
			&=|\tilde{x}_{i,\mathrm{st}}|^{2}+|\tilde{y}_{i,\mathrm{st}}|^{2}+\tilde{y}_{i,\mathrm{st}}^{\ast}\tilde{x}_{i,\mathrm{st}}+\tilde{x}_{i,\mathrm{st}}^{\ast}\tilde{y}_{i,\mathrm{st}}+((\tilde{x}_{i,\mathrm{st}}+\tilde{y}_{i,\mathrm{st}}) (\tilde{x}_{i,\mathcal{I}}^{\ast}+\tilde{y}_{i,\mathcal{I}}^{\ast})+(\tilde{x}_{i,\mathcal{I}}\\
			&+\tilde{y}_{i,\mathcal{I}}) (\tilde{x}_{i,\mathrm{st}}^{\ast}+\tilde{y}_{i,\mathrm{st}}^{\ast}))\epsilon
		\end{aligned}
	\end{equation*}
	and
	\begin{equation*}
		\begin{aligned}
			&|\hat{x}_{i}|^{2}=|\tilde{x}_{i,\mathrm{st}}|^{2}+(\tilde{x}_{i,\mathrm{st}}^{\ast}\tilde{x}_{i,\mathcal{I}}+\tilde{x}_{i,\mathcal{I}}^{\ast}\tilde{x}_{i,\mathrm{st}})\epsilon, \\
			& |\hat{y}_{i}|^{2}=|\tilde{y}_{i,\mathrm{st}}|^{2}+(\tilde{y}_{i,\mathrm{st}}^{\ast}\tilde{y}_{i,\mathcal{I}}+\tilde{y}_{i,\mathcal{I}}^{\ast}\tilde{y}_{i,\mathrm{st}})\epsilon, \\
			&\hat{y}_{i}^{\ast}\hat{x}_{i}=\tilde{y}_{i,\mathrm{st}}^{\ast}\tilde{x}_{i,\mathrm{st}}+(\tilde{y}_{i,\mathrm{st}}^{\ast}\tilde{x}_{i,\mathcal{I}}+\tilde{y}_{i,\mathcal{I}}^{\ast}\tilde{x}_{i,\mathrm{st}})\epsilon, \\
			&\hat{x}_{i}^{\ast}\hat{y}_{i}=\tilde{x}_{i,\mathrm{st}}^{\ast}\tilde{y}_{i,\mathrm{st}}+(\tilde{x}_{i,\mathrm{st}}^{\ast}\tilde{y}_{i,\mathcal{I}}+\tilde{x}_{i,\mathcal{I}}^{\ast}\tilde{y}_{i,\mathrm{st}})\epsilon.
		\end{aligned}
	\end{equation*}
	It is easily verified that
	\begin{equation*}
		|\hat{x}_{i}+\hat{y}_{i}|^{2}=|\hat{x}_{i}|^{2}+|\hat{y}_{i}|^{2}+\hat{y}_{i}^{\ast}\hat{x}_{i}+\hat{x}_{i}^{\ast}\hat{y}_{i}.
	\end{equation*}
	This completes the proof.
\end{proof}

For convenience in the numerical experiments, the $2^{R}$-norm of a dual quaternion vector $\hat{\mathbf{x}} \in \mathbb{DQ}^{n \times 1}$ is defined as
$$
\|\hat{\mathbf{x}}\|_{2^{R}}=\sqrt{\left\|\tilde{\mathbf{x}}_{\mathrm{st}}\right\|_{2}^{2}+\left\|\tilde{\mathbf{x}}_{\mathcal{I}}\right\|_{2}^{2}},
$$
and the $F^{R}$-norm of a dual quaternion matrix $\hat{A} \in \mathbb{DQ}^{m \times n}$ is defined as
\begin{equation*}
\|\hat{A}\|_{F^{R}}=\sqrt{\|\tilde{A}_{\mathrm{st}}\|_{F}^{2}+\|\tilde{A}_{\mathcal{I}}\|_{F}^{2}}.
\end{equation*}
The $2^{R}$-norm and $F^{R}$-norm here are not norms since they do not satisfy the scaling condition of norms. For more details, please refer to \cite{cui2023power}.

For a given dual quaternion matrix $\hat{A} \in \mathbb{DQ}^{n \times n}$. If there are $\hat{\lambda} \in \mathbb{DQ}$, $\hat{\mathbf{x}} \in \mathbb{DQ}^{n\times 1}$, where $\hat{\mathbf{x}}$ is appreciable, such that
$$
\hat{A} \hat{\mathbf{x}}=\hat{\lambda}\hat{\mathbf{x}},
$$
then we say that $\hat{\lambda}$ is a left eigenvalue of $A$, with $\hat{\mathbf{x}}$ as an associated left eigenvector. If
\begin{equation*}
	\hat{A} \hat{\mathbf{x}}=\hat{\mathbf{x}} \hat{\lambda},
\end{equation*}
then we say that $\hat{\lambda}$ is a right eigenvalue of $A$, with $\hat{\mathbf{x}}$ as an associated right eigenvector. Based on Theorem 4.2 in \cite{qi2021eigenvalues}, a dual quaternion Hermitian matrix $\hat{A} \in \mathbb{DQ}^{n \times n}$ has exactly $n$ eigenvalues (with associated $n$ eigenvectors), which are dual numbers.

\section{Dual Representation of a Dual Quaternion Matrix}

\noindent In this section, we put forward a dual matrix isomorphic representation form of dual quaternion matrices.

\begin{theorem}\label{theorepre}
	For any dual quaternion matrix $\hat{A}=\tilde{A}_{\mathrm{st}}+\tilde{A}_{\mathcal{I}}\epsilon\in \mathbb{DQ}^{m \times n}$, where $\tilde{A}_{\mathrm{st}}=A_{1}+A_{2}\mathbf{i}+A_{3}\mathbf{j}+A_{4}\mathbf{k}, \tilde{A}_{\mathcal{I}}=\bar{A}_{1}+\bar{A}_{2}\mathbf{i}+\bar{A}_{3}\mathbf{j}+\bar{A}_{4}\mathbf{k} \in \mathbb{Q}^{m \times n}$ and $A_{i}, \bar{A}_{i} \in \mathbb{R}^{m \times n},~ i=1,2,3,4$. Define the map
		\begin{align}
			\sigma : \ & \mathbb{DQ}^{m\times n} \rightarrow \mathbb{D}^{4m\times 4n}, \nonumber \\
			& \hat{A} \mapsto \hat{A}^{\sigma}=\begin{bmatrix}
				A_{1}+\bar{A}_{1}\epsilon & A_{3}+\bar{A}_{3}\epsilon & A_{2}+\bar{A}_{2}\epsilon & A_{4}+\bar{A}_{4}\epsilon\\
				-A_{3}-\bar{A}_{3}\epsilon & A_{1}+\bar{A}_{1}\epsilon & A_{4}+\bar{A}_{4}\epsilon & -A_{2}-\bar{A}_{2}\epsilon\\
				-A_{2}-\bar{A}_{2}\epsilon & -A_{4}-\bar{A}_{4}\epsilon & A_{1}+\bar{A}_{1}\epsilon & A_{3}+\bar{A}_{3}\epsilon\\
				-A_{4}-\bar{A}_{4}\epsilon & A_{2}+\bar{A}_{2}\epsilon & -A_{3}-\bar{A}_{3}\epsilon & A_{1}+\bar{A}_{1}\epsilon
			\end{bmatrix}. \label{dual representation1}
		\end{align}
		The map $\sigma$ is an isomorphic mapping from $\mathbb{DQ}^{m \times n}$ to $\mathbb{D}^{4m \times 4n}$.
\end{theorem}

\begin{proof}
		It is easy to know that $\sigma$ is a one-to-one mapping from $\mathbb{DQ}^{m \times n}$ to $\mathbb{D}^{4m \times 4n}$. Then, for any dual quaternion matrix $\hat{C}=\tilde{C}_{\mathrm{st}}+\tilde{C}_{\mathcal{I}\epsilon} \in \mathbb{DQ}^{n \times s}$, where $\tilde{C}_{\mathrm{st}}=C_{1}+C_{2}\mathbf{i}+C_{3}\mathbf{j}+C_{4}\mathbf{k},\tilde{C}_{\mathcal{I}}=\bar{C}_{1}+\bar{C}_{2}\mathbf{i}+\bar{C}_{3}\mathbf{j}+\bar{C}_{4}\mathbf{k} \in \mathbb{Q}^{n \times s}$, we have
		\begin{equation}\label{lemma1}
			\hat{A}\hat{C}=(\tilde{A}_{\mathrm{st}}+\tilde{A}_{\mathcal{I}}\epsilon)(\tilde{C}_{\mathrm{st}}+\tilde{C}_{\mathcal{I}}\epsilon)=\tilde{A}_{\mathrm{st}}\tilde{C}_{\mathrm{st}}+(\tilde{A}_{\mathrm{st}}\tilde{C}_{\mathcal{I}}+\tilde{A}_{\mathcal{I}}\tilde{C}_{\mathrm{st}})\epsilon.
		\end{equation}
		On the one hand, consider the standard part of the equality (\ref{lemma1}), we have
		\begin{align*}
			\tilde{A}_{\mathrm{st}}\tilde{C}_{\mathrm{st}}&=(A_{1}+A_{2}\mathbf{i}+A_{3}\mathbf{j}+A_{4}\mathbf{k})(C_{1}+C_{2}\mathbf{i}+C_{3}\mathbf{j}+C_{4}\mathbf{k}) \\
			&=A_{1}C_{1}-A_{2}C_{2}-A_{3}C_{3}-A_{4}C_{4}+(A_{1}C_{2}+A_{2}C_{1}+A_{3}C_{4}-A_{4}C_{3})\mathbf{i} \\
			& + (A_{1}C_{3}-A_{2}C_{4}+A_{3}C_{1}+A_{4}C_{2})\mathbf{j}+(A_{1}C_{4}+A_{2}C_{3}-A_{3}C_{2}+A_{4}C_{1})\mathbf{k},
		\end{align*}
		and consider the dual part of the equality (\ref{lemma1}), we obtain
		\begin{align*}
			\tilde{A}_{\mathrm{st}}\tilde{C}_{\mathcal{I}}+\tilde{A}_{\mathcal{I}}\tilde{C}_{\mathrm{st}}&=(A_{1}+A_{2}\mathbf{i}+A_{3}\mathbf{j}+A_{4}\mathbf{k})(\bar{C}_{1}+\bar{C}_{2}\mathbf{i}+\bar{C}_{3}\mathbf{j}+\bar{C}_{4}\mathbf{k}) \\
			&+(\bar{A}_{1}+\bar{A}_{2}\mathbf{i}+\bar{A}_{3}\mathbf{j}+\bar{A}_{4}\mathbf{k})(C_{1}+C_{2}\mathbf{i}+C_{3}\mathbf{j}+C_{4}\mathbf{k}) \\
			&=A_{1}\bar{C}_{1}-A_{2}\bar{C}_{2}-A_{3}\bar{C}_{3}-A_{4}\bar{C}_{4}+(A_{1}\bar{C}_{2}+A_{2}\bar{C}_{1}+A_{3}\bar{C}_{4}-A_{4}\bar{C}_{3})\mathbf{i} \\
			& + (A_{1}\bar{C}_{3}-A_{2}\bar{C}_{4}+A_{3}\bar{C}_{1}+A_{4}\bar{C}_{2})\mathbf{j}+(A_{1}\bar{C}_{4}+A_{2}\bar{C}_{3}-A_{3}\bar{C}_{2}+A_{4}\bar{C}_{1})\mathbf{k} \\
			&+\bar{A}_{1}C_{1}-\bar{A}_{2}C_{2}-\bar{A}_{3}C_{3}-\bar{A}_{4}C_{4}+(\bar{A}_{1}C_{2}+\bar{A}_{2}C_{1}+\bar{A}_{3}C_{4}-\bar{A}_{4}C_{3})\mathbf{i} \\
			& + (\bar{A}_{1}C_{3}-\bar{A}_{2}C_{4}+\bar{A}_{3}C_{1}+\bar{A}_{4}C_{2})\mathbf{j}+(\bar{A}_{1}C_{4}+\bar{A}_{2}C_{3}-\bar{A}_{3}C_{2}+\bar{A}_{4}C_{1})\mathbf{k} \\
			&=A_{1}\bar{C}_{1}-A_{2}\bar{C}_{2}-A_{3}\bar{C}_{3}-A_{4}\bar{C}_{4}+\bar{A}_{1}C_{1}-\bar{A}_{2}C_{2}-\bar{A}_{3}C_{3}-\bar{A}_{4}C_{4} \\
			& + (A_{1}\bar{C}_{2}+A_{2}\bar{C}_{1}+A_{3}\bar{C}_{4}-A_{4}\bar{C}_{3}+\bar{A}_{1}C_{2}+\bar{A}_{2}C_{1}+\bar{A}_{3}C_{4}-\bar{A}_{4}C_{3})\mathbf{i} \\
			&+(A_{1}\bar{C}_{3}-A_{2}\bar{C}_{4}+A_{3}\bar{C}_{1}+A_{4}\bar{C}_{2}+\bar{A}_{1}C_{3}-\bar{A}_{2}C_{4}+\bar{A}_{3}C_{1}+\bar{A}_{4}C_{2})\mathbf{j} \\
			& + (A_{1}\bar{C}_{4}+A_{2}\bar{C}_{3}-A_{3}\bar{C}_{2}+A_{4}\bar{C}_{1}+\bar{A}_{1}C_{4}+\bar{A}_{2}C_{3}-\bar{A}_{3}C_{2}+\bar{A}_{4}C_{1})\mathbf{k},
		\end{align*}
		which implies that $(\hat{A}\hat{C})_{[1,1]}^{\sigma}=A_{1}C_{1}-A_{2}C_{2}-A_{3}C_{3}-A_{4}C_{4}+(A_{1}\bar{C}_{1}-A_{2}\bar{C}_{2}-A_{3}\bar{C}_{3}-A_{4}\bar{C}_{4}+\bar{A}_{1}C_{1}-\bar{A}_{2}C_{2}-\bar{A}_{3}C_{3}-\bar{A}_{4}C_{4})\epsilon$, where the symbols $(\hat{A}^{\sigma})_{[i,j]},i,j=1,2,3,4$ denote the block matrix at the position of row $i$ and column $j$ of $\hat{A}^{\sigma}$. On the other hand, it is evident that
		\begin{equation}\label{dual representation2}
			\hat{C}^{\sigma}=
			\begin{bmatrix}
				C_{1}+\bar{C}_{1}\epsilon & C_{3}+\bar{C}_{3}\epsilon & C_{2}+\bar{C}_{2}\epsilon & C_{4}+\bar{C}_{4}\epsilon\\
				-C_{3}-\bar{C}_{3}\epsilon & C_{1}+\bar{C}_{1}\epsilon & C_{4}+\bar{C}_{4}\epsilon & -C_{2}-\bar{C}_{2}\epsilon\\
				-C_{2}-\bar{C}_{2}\epsilon & -C_{4}-\bar{C}_{4}\epsilon & C_{1}+\bar{C}_{1}\epsilon & C_{3}+\bar{C}_{3}\epsilon\\
				-C_{4}-\bar{C}_{4}\epsilon & C_{2}+\bar{C}_{2}\epsilon & -C_{3}-\bar{C}_{3}\epsilon & C_{1}+\bar{C}_{1}\epsilon
			\end{bmatrix} \in \mathbb{D}^{4n \times 4s},
		\end{equation}
		from the formulas \eqref{dual representation1} and \eqref{dual representation2}, we get $(\hat{A}^{\sigma}\hat{C}^{\sigma})_{[1,1]}=A_{1}C_{1}-A_{2}C_{2}-A_{3}C_{3}-A_{4}C_{4}+(A_{1}\bar{C}_{1}+\bar{A}_{1}C_{1}-A_{2}\bar{C}_{2}-\bar{A}_{2}C_{2}-A_{3}\bar{C}_{3}-\bar{A}_{3}C_{3}-A_{4}\bar{C}_{4}-\bar{A}_{4}C_{4})\epsilon$, then $(\hat{A}\hat{C})_{[1,1]}^{\sigma}=(\hat{A}^{\sigma}\hat{C}^{\sigma})_{[1,1]}$. Similarity, we can prove $(\hat{A}\hat{C})_{[i,j]}^{\sigma}=(\hat{A}^{\sigma}\hat{C}^{\sigma})_{[i,j]},i,j=1,2,3,4$, thus
		$$(\hat{A}\hat{C})^{\sigma}=\hat{A}^{\sigma}\hat{C}^{\sigma}.$$
		Additionally, it is obvious that $(\hat{A}+\hat{B})^{\sigma}=\hat{A}^{\sigma}+\hat{B}^{\sigma}$ holds for any dual quaternion matrix $\hat{B}=\tilde{B}_{\mathrm{st}}+\tilde{B}_{\mathcal{I}}\in \mathbb{DQ}^{m \times n}$. Therefore, we can conclude that $\sigma$ is an isomorphic mapping from $\mathbb{DQ}^{m \times n}$ to $\mathbb{D}^{4m \times 4n}$.
	\end{proof}

For matrix representations of dual quaternions, one may also refer to \cite{demir2007matrix}. In the sequel, we list some of the properties of $\hat{A}^{\sigma}$ as follows.

\begin{lemma} \label{dualre}
Let $\hat{A},\hat{B}\in \mathbb{DQ}^{m \times n}$, $\hat{C} \in \mathbb{DQ}^{n \times s}$ and $a \in \mathbb{D}$. Then the following properties hold.

\noindent $(1)$ $\hat{A}=\hat{B} \Longleftrightarrow \hat{A}^{\sigma}=\hat{B}^{\sigma}$.

\noindent $(2)$ $(\hat{A}+\hat{B})^{\sigma}=\hat{A}^{\sigma}+\hat{B}^{\sigma}$, $(a\hat{A})^{\sigma}=a\hat{A}^{\sigma}$.

\noindent $(3)$ $(\hat{A}\hat{C})^{\sigma}=\hat{A}^{\sigma}\hat{C}^{\sigma}$.

\end{lemma}

Let $\hat{A}_{c}^{\sigma}$ denote the first column block of $A^{\sigma}$, that is
\begin{equation}\label{column}
\hat{A}_{c}^{\sigma}=
\begin{bmatrix}
	A_{1}+\bar{A}_{1}\epsilon \\
	-A_{3}-\bar{A}_{3}\epsilon \\
	-A_{2}-\bar{A}_{2}\epsilon \\
	-A_{4}-\bar{A}_{4}\epsilon 
\end{bmatrix} \in \mathbb{D}^{4m \times n}.
\end{equation}
From this notation and Lemma \ref{dualre}, we have the following results.

\begin{lemma} \label{dualre_lemma}
Let $\hat{A},\hat{B}\in \mathbb{DQ}^{m \times n}$, $\hat{C} \in \mathbb{DQ}^{n \times s}$ and $a \in \mathbb{D}$. Then the following properties hold.

\noindent $(1)$ $(\hat{A}+\hat{B})_{c}^{\sigma}=\hat{A}_{c}^{\sigma}+\hat{B}_{c}^{\sigma}$.

\noindent $(2)$ $(a\hat{A})_{c}^{\sigma}=a\hat{A}_{c}^{\sigma}$.

\noindent $(3)$ $(\hat{A}\hat{C})_{c}^{\sigma}=\hat{A}^{\sigma}\hat{C}_{c}^{\sigma}$.
\end{lemma}

\begin{proof}
We only prove the property $(3)$. This property can be obtained from the proof of Theorem \ref{theorepre}.
\end{proof}

From (\ref{column}) and Lemma \ref{dualre_lemma}, $\hat{A}_{c}^{\sigma}$ contains all the information of $\hat{A}^{\sigma}$. Algorithms, constructed by $\hat{A}^{\sigma}$ instead of $\hat{A}^{\sigma}$, have higher computational and storage efficiency, this is we called the structure-preserving strategy. The algorithms based on the real or complex structure-preserving strategy can be also found in \cite{jia2013new,yu2024new}.

\section{Rayleigh Quotient Iteration for Computing the Eigenvalue of a Dual Quaternion Hermitian Matrix}\label{section3}

\noindent In this section, we study the Rayleigh quotient iteration (RQI) for computing the eigenvalue with associated eigenvector of a dual quaternion Hermitian matrix.

The following theorem gives the form of the eigenvalues of a nonsingular dual quaternion matrix.

\begin{theorem}\label{inverse_eigenvalue}
Suppose that $\hat{A} \in \mathbb{DQ}^{n \times n}$ is a non-singular Hermitian matrix, $\lambda=\lambda_{\mathrm{st}}+\lambda_{\mathcal{I}}\epsilon \in \mathbb{D}$ is an appreciable right eigenvalue of $A$, then $\lambda^{-1}=1 / \lambda_{\mathrm{st}}-\lambda_{\mathcal{I}}\epsilon / \lambda_{\mathrm{st}}^{2}$ is a right eigenvalue of $\hat{A}^{-1}$.
\end{theorem}

\begin{proof}
By Theorem 4.1 in \cite{qi2021eigenvalues},  there exists a unitary matrix $\hat{U} \in \mathbb{DQ}^{n \times n}$ and a diagonal matrix $\Sigma \in \mathbb{D}^{n \times n}$  such that
\begin{equation*}
	\hat{A}=\hat{U} \Sigma \hat{U}^{*},
\end{equation*}
where
$\Sigma=\operatorname{diag}\left(\lambda_{1}, \ldots, \lambda_{n}\right) \in \mathbb{D}^{n \times n}$ is a diagonal dual matrix, and $\lambda_{i}=\lambda_{i}+\lambda_{i,\mathcal{I}}\epsilon, i=1, \ldots, n$ are all appreciable eigenvalues of $\hat{A}$. Then we have 
\begin{equation*}
	\begin{aligned}
		\hat{A}^{-1} & = \left(\hat{U} \Sigma \hat{U}^{*}\right)^{-1}=\hat{U} \Sigma^{-1} \hat{U}^{*}\\
		& = \hat{U} \operatorname{diag}\left(\lambda_{1}^{-1}, \lambda_{2}^{-1}, \ldots, \lambda_{n}^{-1}\right) \hat{U}^{*}\\
		& = \hat{U} \operatorname{diag}\left(\frac{1}{\lambda_{1}}-\frac{\lambda_{1,\mathcal{I}}}{\lambda_{1}^{2}} \epsilon, \frac{1}{\lambda_{2}}-\frac{\lambda_{2,\mathcal{I}}}{\lambda_{2}^{2}} \epsilon, \ldots, \frac{1}{\lambda_{n}}-\frac{\lambda_{n,\mathcal{I}}}{\lambda_{n}^{2}} \epsilon\right) \hat{U}^{*},
	\end{aligned}
\end{equation*}
which implies that $\lambda_{i}^{-1}=\dfrac{1}{\lambda_{i}}-\dfrac{\lambda_{i,\mathcal{I}}}{\lambda_{i}^{2}} \epsilon, i=1, \ldots, n$ are the right eigenvalues of $\hat{A}^{-1}$. 
\end{proof}

Next, we review an important proposition in \cite{qi2021eigenvalues}.

\begin{proposition}[\cite{qi2021eigenvalues}]\label{Rayleigh}
Suppose that $\hat{\lambda} \in \mathbb{DQ}$  is a right eigenvalue of $\hat{A} \in \mathbb{DQ}^{n \times n}$, with associated right eigenvector  $\hat{\mathbf{x}} \in \mathbb{DQ}^{n\times 1}$. Then
\begin{equation*}
	\hat{\lambda}=\dfrac{\hat{\mathbf{x}}^{\ast} \hat{A} \hat{\mathbf{x}}}{\hat{\mathbf{x}}^{\ast} \hat{\mathbf{x}}}.
\end{equation*}
\end{proposition}

Assume that the unit dual quaternion vector $\hat{\mathbf{u}}_{k}$ is a reasonably approximation to $\hat{\mathbf{x}}$, then $\theta_{k}=\hat{\mathbf{u}}_{k}^{\ast}\hat{A}\hat{\mathbf{u}}_{k}$ is a good approximation to $\hat{\lambda}$ too.

\begin{theorem}[Unitary Decomposition \cite{qi2021eigenvalues}]\label{unitary}
	Suppose that $\hat{A}=\tilde{A}_{s t}+\tilde{A}_{\mathcal{I}}\epsilon \in \mathbb{DQ}^{n \times n}$ is a Hermitian matrix. Then there are unitary matrix $\hat{U} \in \mathbb{DQ}^{n \times n}$ and a diagonal matrix $\Sigma \in \mathbb{D}^{n \times n}$ such that $\Sigma=\hat{U}^{\ast} \hat{A} \hat{U}$, where
	\begin{equation*}
		\Sigma = \operatorname{diag}\left(\lambda_{1,\mathrm{st}}+\lambda_{1,\mathcal{I}_{1}} \epsilon, \cdots, \lambda_{1,\mathrm{st}}+\lambda_{1, \mathcal{I}_{k_{1}}} \epsilon, \lambda_{2,\mathrm{st}}+\lambda_{2,\mathcal{I}_{1}} \epsilon, \cdots, \lambda_{r,\mathrm{st}}+\lambda_{r, \mathcal{I}_{k_{r}}} \epsilon\right),
	\end{equation*}
	with the diagonal entries of $\Sigma$ being $n$ eigenvalues of $\hat{A}$,
	\begin{equation*}
		\hat{A} \hat{\mathbf{u}}_{i, j}=\hat{\mathbf{u}}_{i, j}\left(\lambda_{i,\mathrm{st}}+\lambda_{i,\mathcal{I}_{j}} \epsilon\right),
	\end{equation*}
	for $j=1, \cdots, k_{i}$ and $i=1, \cdots, r$, $\hat{U}=\left(\hat{\mathbf{u}}_{1,1}, \cdots, \hat{\mathbf{u}}_{1, k_{1}}, \cdots, \hat{\mathbf{u}}_{r, k_{r}}\right)$, $\lambda_{1,\mathrm{st}}>\lambda_{2,\mathrm{st}}>\cdots>\lambda_{r,\mathrm{st}}$  are real numbers, $\lambda_{i,\mathrm{st}}$ is a $k_{i}$-multiple right eigenvalue of $\tilde{A}_{\mathrm{st}}$, $\lambda_{i, \mathcal{I}_{1}} \geq \lambda_{i,\mathcal{I}_{2}} \geq \cdots \geq \lambda_{i, \mathcal{I}_{k_{i}}}$ are also real numbers. Counting possible multiplicities $\lambda_{i,\mathcal{I}_{j}}$, the form $\Sigma$ is unique.
\end{theorem}

Below, we define the Rayleigh quotient for Hermitian matrices and study its fundamental properties.

\begin{definition}
	If $\hat{A}$ is a dual quaternion Hermitian matrix, then the Rayleigh Quotient $\hat{\theta}$ is the function which assigns to any appreciable dual quaternion vector $\hat{\mathbf{v}}$, i.e.
	\begin{equation*}
		\theta(\hat{A},\hat{\mathbf{x}})=\dfrac{\hat{\mathbf{v}}^{\ast} \hat{A} \hat{\mathbf{v}}}{\hat{\mathbf{v}}^{\ast} \hat{\mathbf{v}}}.
	\end{equation*}
\end{definition}

\begin{proposition}
	$\theta(\hat{A},\hat{\mathbf{v}})=\theta(\hat{A},\hat{\mathbf{v}}\hat{k})$, where $\hat{\mathbf{v}}$ is an appreciable dual quaternion vector and $\hat{k}$ is an appreciable dual quaternion.
\end{proposition}

\begin{proof}
	\begin{equation*}
		\theta(\hat{A},\hat{\mathbf{v}}\hat{k})=\dfrac{(\hat{\mathbf{v}}\hat{k})^{\ast} \hat{A} (\hat{\mathbf{v}}\hat{k})}{(\hat{\mathbf{v}}\hat{k})^{\ast} (\hat{\mathbf{v}}\hat{k})}=\dfrac{\hat{k}^{\ast}\hat{\mathbf{v}}^{\ast} \hat{A} \hat{\mathbf{v}}\hat{k}}{\hat{k}^{\ast}\hat{\mathbf{v}}^{\ast} \hat{\mathbf{v}}\hat{k}}=\dfrac{\hat{k}^{\ast}\hat{k}\hat{\mathbf{x}}^{\ast} \hat{A} \hat{\mathbf{v}}}{\hat{k}^{\ast}\hat{k}\hat{\mathbf{v}}^{\ast} \hat{\mathbf{v}}}=\dfrac{\hat{\mathbf{v}}^{\ast} \hat{A} \hat{\mathbf{v}}}{\hat{\mathbf{v}}^{\ast} \hat{\mathbf{v}}}=\theta(\hat{A},\hat{\mathbf{v}}).
	\end{equation*}
\end{proof}

If $\hat{\mathbf{v}} \in \mathbb{DQ}^{n}$ is a nonzero vector, we can write it in coordinates with respect to the orthogonal basis $\left\{\hat{\mathbf{u}}_{j}\right\}_{j=1, \ldots, n}$ as
\begin{equation*}
	\hat{\mathbf{v}}=\sum_{j=1}^{n} \hat{\mathbf{u}}_{j} \hat{x}_{j}, \quad \hat{x}_{j} \in \mathbb{DQ}.
\end{equation*}

\begin{proposition}\label{mean} 
	Let $\hat{A}\in \mathbb{DQ}^{n\times n}$, where $\hat{A}$ is Hermitian. Then the Rayleigh quotient equals the weighted mean
	\begin{equation*}
		\theta(\hat{A},\hat{\mathbf{v}})=\frac{\sum_{j} \lambda_{j}\hat{x}_{j}^{\ast} \hat{x}_{j}}{\sum_{j}\hat{x}_{j}^{\ast} \hat{x}_{j}},
	\end{equation*}
	where the dual numbers $\lambda_{j}$ are the right eigenvalues of the dual quaternion Hermitian matrix $\hat{A}$.
\end{proposition} 

\begin{proof}
	According to Theorem \ref{unitary}, there exists a unitary matrix $U=\left(\hat{\mathbf{u}}_{1},\hat{\mathbf{u}}_{2},\cdots,\hat{\mathbf{u}}_{n}\right)$ such that
	\begin{equation*}
		\hat{A}\hat{U}=\hat{U}\mathrm{diag}\left(\lambda_{1},\lambda_{2},\cdots,\lambda_{n}\right), 
	\end{equation*}
	where $\hat{\mathbf{u}}_{1},\hat{\mathbf{u}}_{2},\cdots,\hat{\mathbf{u}}_{n}$ constitute an orthogonal basis of $\mathbb{DQ}^{n}$. Since $\hat{\mathbf{u}}_{i}^{\ast} \hat{\mathbf{u}}_{j}=\delta_{ij}$, we have
	\begin{equation*}
		\hat{\mathbf{v}}^{\ast} \hat{\mathbf{v}}=\left(\sum_{i} \hat{x}_{i}^{\ast} \hat{\mathbf{u}}_{i}^{\ast}\right)\left(\sum_{j} \hat{\mathbf{u}}_{j} \hat{x}_{j}\right)=\sum_{i, j} \hat{x}_{i}^{\ast} \hat{\mathbf{u}}_{i}^{\ast} \hat{\mathbf{u}}_{j} \hat{x}_{j}=\sum_{j} \hat{x}_{j}^{\ast} \hat{x}_{j}.
	\end{equation*}
	Analogously, since $\hat{A} \hat{\mathbf{u}}_{j}=\hat{\mathbf{u}}_{j} \lambda_{j}$, then we have
	\begin{equation*}
		\hat{\mathbf{v}}^{\ast} \hat{A} \hat{\mathbf{v}}=\left(\sum_{i} \hat{x}_{i}^{\ast} \hat{\mathbf{u}}_{i}^{\ast}\right)\left(\sum_{j} \hat{\mathbf{u}}_{j} \lambda_{j} \hat{x}_{j}\right)=\sum_{j} \hat{x}_{i}^{\ast} \lambda_{j} \hat{x}_{j}=\sum_{j} \lambda_{j}\hat{x}_{i}^{\ast} \hat{x}_{j}.
	\end{equation*}
This completes the proof.
\end{proof}

\begin{proposition}\label{maxmin}
	Let $\hat{A}\in \mathbb{DQ}^{n\times n}$, where $\hat{A}$ is Hermitian. If $\lambda_{1},\lambda_{2},\cdots,\lambda_{n}$ are generalized eigenvalues of $(A,B)$ with $\lambda_{1}\leq\lambda_{2}\leq \cdots \lambda_{n}$, then
	\begin{equation*}
		\lambda_{1} \leq \theta(\hat{A},\hat{\mathbf{v}}) \leq \lambda_{n},
	\end{equation*}
where $\mathbf{v}$ is an appreciable dual quaternion vector.
\end{proposition}
\begin{proof}
	According to Proposition \ref{mean}, it holds that 
	\begin{equation*}
		\lambda_{1}\left(\sum_{j} \hat{x}_{j}^{\ast} \hat{x}_{j}\right) \leq \sum_{j} \lambda_{j}\hat{x}_{j}^{\ast} \hat{x}_{j} \leq \lambda_{n}\left(\sum_{j} \hat{x}_{j}^{\ast} \hat{x}_{j}\right),
	\end{equation*}
	then 
	\begin{equation*}
		\lambda_{1} \leq \frac{\sum_{j} \lambda_{j}\hat{x}_{j}^{\ast} \hat{x}_{j}}{\sum_{j}\hat{x}_{j}^{\ast} \hat{x}_{j}} \leq \lambda_{n}.
	\end{equation*}
	This is complete the proof.
\end{proof}

\begin{corollary}
	The function $\theta(\hat{A},\hat{\mathbf{v}})$ attains its minimum value at the smallest right eigenvalue of $\hat{A}$, and its maximum value at the largest right eigenvalue of $\hat{A}$.
\end{corollary}

When seeking an eigenvector associated with the eigenvalue near some prescribed target, we often adopt the practically useful and effective Rayleigh quotient iteration (RQI) method, which essentially possesses a shift-and-invert transformation to the currently available approximate eigenvector. That is, at each iterate the following pair of equations
\begin{equation*}
\hat{\theta}_{k}=\hat{\mathbf{u}}_{k}^{\ast}\hat{A}\hat{\mathbf{u}}_{k},
~ (\hat{A}-\hat{\theta}_{k} I)\hat{\mathbf{w}}_{k+1}=\hat{\mathbf{u}}_{k},
~ \hat{\mathbf{u}}_{k+1}=\frac{\hat{\mathbf{w}}_{k+1}}{\|\hat{\mathbf{w}}_{k+1}\|_{2}},
\end{equation*}
where $\hat{\mathbf{u}}_{k}$ is the currently available approximate eigenvector with $\|\hat{\mathbf{u}}_{k}\|_{2}=1$, and $\hat{\theta}_{k}$ is the corresponding Rayleigh quotient, need to be solved. This process is repeated until convergent or the maximal iteration number is reached.

\subsection{On the solution of the linear system $(\hat{A}-\hat{\theta} I)\hat{\mathbf{w}}=\hat{\mathbf{u}}$}

\noindent The key step of the Rayleigh quotient iteration is how to effectively solve the linear system $(\hat{A}-\hat{\theta} I)\hat{\mathbf{w}}=\hat{\mathbf{u}}$. From Lemma \ref{dualre_lemma}, we have the following equation
\begin{align}
(\hat{A}-\hat{\theta} I)\hat{\mathbf{w}}=\hat{\mathbf{u}} ~ & \Longleftrightarrow ~ \left[(\hat{A}-\hat{\theta} I)\hat{\mathbf{w}}\right]_{c}^{\sigma}=\hat{\mathbf{u}}_{c}^{\sigma} \label{equation1} \\
& \Longleftrightarrow ~ (\hat{A}-\hat{\theta} I)^{\sigma}\hat{\mathbf{w}}_{c}^{\sigma}=\hat{\mathbf{u}}_{c}^{\sigma}. \nonumber
\end{align}
The equation (\ref{equation1}) is derived by the dual representation and the structure-preserving strategy. In this context, we equivalently transform the linear system over dual quaternions into a linear system over dual numbers.

To solve $(\hat{A}-\hat{\theta} I)^{\sigma}\hat{\mathbf{w}}_{c}^{\sigma}=\hat{\mathbf{u}}_{c}^{\sigma}$, let $\hat{A}-\hat{\theta} I=\tilde{A}_{\mathrm{st}}+\tilde{A}_{\mathcal{I}}\epsilon$, $\hat{\mathbf{w}}_{c}^{\sigma}=\tilde{\mathbf{w}}_{\mathrm{st}}+\tilde{\mathbf{w}}_{\mathcal{I}}\epsilon$, $\hat{\mathbf{u}}_{c}^{\sigma}=\tilde{\mathbf{u}}_{\mathrm{st}}+\tilde{\mathbf{u}}_{\mathcal{I}}\epsilon$, where $\tilde{A}_{\mathrm{st}}, \tilde{A}_{\mathcal{I}} \in \mathbb{R}^{4n\times 4n},~\tilde{\mathbf{w}}_{\mathrm{st}},\tilde{\mathbf{w}}_{\mathcal{I}},\tilde{\mathbf{u}}_{\mathrm{st}},\tilde{\mathbf{u}}_{\mathcal{I}} \in \mathbb{R}^{4n\times 1}$, then we have
\begin{align}
(\hat{A}-\hat{\theta} I)^{\sigma}\hat{\mathbf{w}}_{c}^{\sigma}=\hat{\mathbf{u}}_{c}^{\sigma} ~ & \Longleftrightarrow ~ (\tilde{A}_{\mathrm{st}}+\tilde{A}_{\mathcal{I}}\epsilon)(\tilde{\mathbf{w}}_{\mathrm{st}}+\tilde{\mathbf{w}}_{\mathcal{I}}\epsilon)=\tilde{\mathbf{u}}_{\mathrm{st}}+\tilde{\mathbf{u}}_{\mathcal{I}}\epsilon \nonumber \\
& \Longleftrightarrow ~ \tilde{A}_{\mathrm{st}}\tilde{\mathbf{w}}_{\mathrm{st}}+\tilde{A}_{\mathrm{st}}\tilde{\mathbf{w}}_{\mathcal{I}}\epsilon+\tilde{A}_{\mathcal{I}}\mathbf{x}_{\mathrm{st}}\epsilon=\tilde{\mathbf{u}}_{\mathrm{st}}+\tilde{\mathbf{u}}_{\mathcal{I}}\epsilon \nonumber \\
& \Longleftrightarrow ~ \tilde{A}_{\mathrm{st}}\tilde{\mathbf{w}}_{\mathrm{st}}+(\tilde{A}_{\mathrm{st}}\tilde{\mathbf{w}}_{\mathcal{I}}+\tilde{A}_{\mathcal{I}}\tilde{\mathbf{w}}_{\mathrm{st}})\epsilon=\tilde{\mathbf{u}}_{\mathrm{st}}+\tilde{\mathbf{u}}_{\mathcal{I}}\epsilon. \label{equation2}
\end{align}
Considering the standard and dual parts of equation (\ref{equation2}), one can split it into the following two matrix equations
\begin{equation*}
\left\{\begin{aligned}
	&\tilde{A}_{\mathrm{st}}\tilde{\mathbf{w}}_{\mathrm{st}} =\tilde{\mathbf{u}}_{\mathrm{st}}, \\
	&\tilde{A}_{\mathrm{st}}\tilde{\mathbf{w}}_{\mathcal{I}}+\tilde{A}_{\mathcal{I}}\tilde{\mathbf{w}}_{\mathrm{st}} = \tilde{\mathbf{u}}_{\mathcal{I}},
\end{aligned}\right.
\end{equation*}
which can be equivalently written as
\begin{equation}\label{eqn1}
\begin{bmatrix}
	\tilde{A}_{\mathrm{st}} & \tilde{A}_{\mathcal{I}} \\
	\mathbf{O} & \tilde{A}_{\mathrm{st}}
\end{bmatrix}\begin{bmatrix} 
	\tilde{\mathbf{w}}_{\mathcal{I}} \\
	\tilde{\mathbf{w}}_{\mathrm{st}}
\end{bmatrix}=\begin{bmatrix}
	\tilde{\mathbf{u}}_{\mathcal{I}} \\
	\tilde{\mathbf{u}}_{\mathrm{st}}
\end{bmatrix}.
\end{equation}
The linear system (\ref{eqn1}) is consistent if and only if
\begin{equation*}
\mathrm{rank}\begin{bmatrix}
	\tilde{A}_{\mathrm{st}} &\tilde{A}_{\mathcal{I}}&\tilde{\mathbf{u}}_{\mathcal{I}} \\
	\mathbf{O}&\tilde{A}_{\mathrm{st}}&\tilde{\mathbf{u}}_{\mathrm{st}}
\end{bmatrix}=\mathrm{rank}\begin{bmatrix}
	\tilde{A}_{\mathrm{st}} & \tilde{A}_{\mathcal{I}} \\
	\mathbf{O}&\tilde{A}_{\mathrm{st}}
\end{bmatrix}.
\end{equation*}
In addition, if the linear system (\ref{eqn1}) is not consistent, then one can find a solution $\mathbf{x}^{\star}$ such that
\begin{equation*}
\left\|\begin{bmatrix}
	\tilde{A}_{\mathrm{st}} & \tilde{A}_{\mathcal{I}} \\
	\mathbf{O} & \tilde{A}_{\mathrm{st}}
\end{bmatrix}\begin{bmatrix} 
	\tilde{\mathbf{w}}_{\mathcal{I}}^{\star} \\
	\tilde{\mathbf{w}}_{\mathrm{st}}^{\star}
\end{bmatrix}-\begin{bmatrix}
	\tilde{\mathbf{u}}_{\mathcal{I}} \\
	\tilde{\mathbf{u}}_{\mathrm{st}}
\end{bmatrix}\right\|_{2}=\min\limits_{\mathbf{x}\in \mathbb{R}^{8n\times 1}}\left\|\begin{bmatrix}
	\tilde{A}_{\mathrm{st}} & \tilde{A}_{\mathcal{I}} \\
	\mathbf{O} & \tilde{A}_{\mathrm{st}}
\end{bmatrix}\begin{bmatrix} 
	\tilde{\mathbf{w}}_{\mathcal{I}} \\
	\tilde{\mathbf{w}}_{\mathrm{st}}
\end{bmatrix}-\begin{bmatrix}
	\tilde{\mathbf{u}}_{\mathcal{I}} \\
	\tilde{\mathbf{u}}_{\mathrm{st}}
\end{bmatrix}\right\|_{2},
\end{equation*}
where $\tilde{\mathbf{w}}^{\star}$ is called a least squares solution. In both cases, the general solution to (\ref{eqn1}) can be expressed as
\begin{equation*}
\begin{bmatrix} 
	\tilde{\mathbf{w}}_{\mathcal{I}} \\
	\tilde{\mathbf{w}}_{\mathrm{st}}
\end{bmatrix}=\begin{bmatrix}
	\tilde{A}_{\mathrm{st}} & \tilde{A}_{\mathcal{I}} \\
	\mathbf{O}&\tilde{A}_{\mathrm{st}}
\end{bmatrix}^{\dagger}\begin{bmatrix}
	\tilde{\mathbf{u}}_{\mathcal{I}} \\
	\tilde{\mathbf{u}}_{\mathrm{st}}
\end{bmatrix}+\left(I_{8n}-\begin{bmatrix}
	\tilde{A}_{\mathrm{st}} & \tilde{A}_{\mathcal{I}} \\
	\mathbf{O}&\tilde{A}_{\mathrm{st}}
\end{bmatrix}^{\dagger}\begin{bmatrix}
	\tilde{A}_{\mathrm{st}} & \tilde{A}_{\mathcal{I}} \\
	\mathbf{O}&\tilde{A}_{\mathrm{st}}
\end{bmatrix}\right)\mathbf{z},
\end{equation*}
where $\mathbf{z}$ is a vector of any suitable sizes. Fortunately, it is not necessary to compute the matrix
\begin{equation*}
\begin{bmatrix}
	\tilde{A}_{\mathrm{st}} & \tilde{A}_{\mathcal{I}} \\
	\mathbf{O} & \tilde{A}_{\mathrm{st}}
\end{bmatrix}^{\dagger},
\end{equation*}
explicitly as this could
be rather expensive for large problems. Instead, we just need to carry out the singular value decomposition prior to starting the vector iteration itself. Subsequently, the solution of (\ref{equation2}) can be obtained.

The proposed Rayleigh quotient iteration (RQI) for solving Problem (\ref{problem}) is summarized in Algorithm \ref{RQI}.

\begin{algorithm}[H]
\caption{\ The Rayleigh quotient iteration (RQI).}
\begin{algorithmic}[1]\label{RQI}
	\REQUIRE Given a normalized initial guess $\hat{\mathbf{u}}_{0} \in \mathbb{DQ}^{n\times 1}$ with $\|\hat{\mathbf{u}}_{0}\|_{2}=1$, the maximal iteration number $k_{max}$ and the stopping tolerance $\varepsilon$.
	\ENSURE Eigenvalue $\hat{\theta}_{k}$ and eigenvector $\hat{\mathbf{u}}_{k+1}$.
	\FOR {$k=0,1,2,...,k_{max}$}
	\STATE Compute the Rayleigh quotient $\hat{\theta}_{k}=\hat{\mathbf{u}}_{k}^{\ast}\hat{A}\hat{\mathbf{u}}_{k}$.
	\STATE If $\hat{A}-\hat{\theta}_{k}I$ is singular, then solve $(\hat{A}-\hat{\theta}_{k}I)\hat{\mathbf{u}}_{k+1}=\hat{\mathbf{0}}$ for unit vector $\hat{\mathbf{u}}_{k+1}$ and stop. Otherwise, solve equation $(\hat{A}-\hat{\theta}_{k}I)\hat{\mathbf{w}}_{k+1}=\hat{\mathbf{u}}_{k}$ for $\hat{\mathbf{w}}_{k+1}$. 
	\STATE Normalize $\hat{\mathbf{u}}_{k+1}=\hat{\mathbf{w}}_{k+1} / \|\hat{\mathbf{w}}_{k+1}\|_{2}$.
	\STATE If $\|\hat{A}\hat{\mathbf{u}}_{k+1}-\hat{\mathbf{u}}_{k+1}\hat{\theta}_{k}\|_{2^{R}} \leqslant \varepsilon$, then stop.
	\ENDFOR
\end{algorithmic}
\end{algorithm}

\begin{remark}
In order to avoid the matrix factorization and reduce the total computational cost of the RQI, we may solve the shifted linear system in Step $3$ roughly such that
$$\|(\hat{A}-\hat{\theta}_{k}I)\hat{\mathbf{w}}_{k+1}-\hat{\mathbf{u}}_{k}\|_{2^{R}} \leq \delta$$
by an iteration method, resulting in the so-called inexact Rayleigh quotient iteration.
\end{remark}

\subsection{Convergence Analysis}\label{section4}

\noindent Similar to the Rayleigh quotient iteration for the real matrices, we show the eigenpair $(\hat{\theta}_{k},\hat{\mathbf{u}}_{k})$ converges to a fixed eigenvalue with associated eigenvector. We specifically demonstrate the local convergence property of RQI by examining sequences related to the approximate eigenvectors, residual norms, and Raleigh quotients. These sequences are denoted as $\left\{\mathbf{u}_{k}\right\}_{k=0}^{\infty}$, representing the approximate eigenvectors; $\left\{\left\|\mathbf{r}_{k}\right\|_{2}\right\}_{k=0}^{\infty}$, measuring the residual norms; and $\left\{\theta_{k}\right\}_{k=1}^{\infty}$, characterizing the Raleigh quotients associated with the approximate eigenpairs, which are denoted as $\left(\theta_{k}, \mathbf{u}_{k}\right)$ for $k=1,2,\ldots,n$.

For purposes of analysis, it is convenient to combine lines 3 and 4 of Algorithm \ref{RQI} into
\begin{equation*}
	(\hat{A}-\hat{\theta}_{k}I)\hat{\mathbf{u}}_{k+1}=\hat{\mathbf{u}}_{k}\tau_{k},
\end{equation*}
where $\tau_{k}$ is an appreciable positive dual number which ensures that $\|\hat{\mathbf{u}}_{k+1}\|_{2}=1$.

First of all , we introduce the minimal residual property of the Rayleigh quotient.

\begin{lemma}\label{mini}
	Given a dual quaternion Hermitian matrix $\hat{A}$, for any unit dual quaternion vector $\hat{\mathbf{x}}$ satisfying $\|\hat{\mathbf{x}}\|_{2}=1$, the Rayleigh quotient  $\theta=\hat{\mathbf{x}}^{\ast}\hat{A}\hat{\mathbf{x}}$ satisfies
	\begin{equation*}
		\|(\hat{A}-\theta I) \hat{\mathbf{x}}\|_{2} \leq\|(\hat{A}-\lambda I) \hat{\mathbf{x}}\|_{2}
	\end{equation*}
	for any dual number $\lambda$.
\end{lemma}

\begin{proof}
	\begin{align}
		\|(\hat{A}-\lambda I)\mathbf{x}\|_{2}^{2}&=((\hat{A}-\lambda I)\hat{\mathbf{x}})^{\ast}((\hat{A}-\lambda I)\hat{\mathbf{x}}) \label{mini11} \\
		&=(\hat{\mathbf{x}}^{\ast}\hat{A}-\lambda\hat{\mathbf{x}}^{\ast})(\hat{A}\hat{\mathbf{x}}-\lambda\hat{\mathbf{x}}) \nonumber \\
		&=\hat{\mathbf{x}}^{\ast}\hat{A}\hat{A}\hat{\mathbf{x}}-2\lambda\hat{\mathbf{x}}^{\ast}\hat{A}\hat{\mathbf{x}}+\lambda^{2} \nonumber \\
		&=(\lambda-\hat{\mathbf{x}}^{\ast}\hat{A}\hat{\mathbf{x}})^{2} \nonumber.
	\end{align}
	Equality \eqref{mini11} is obtained by Lemma \ref{induce}. Therefore, the Rayleigh quotient
	$\theta=\hat{\mathbf{x}}^{\ast}\hat{A}\hat{\mathbf{x}}$ minimizes the residual norm $\|(A-\lambda I)\hat{\mathbf{x}}\|_{2}$, proving the minimal residual property of the Rayleigh quotient.
\end{proof}

\begin{theorem}\label{mono}
	Let $\hat{\mathbf{r}}_{k}=(\hat{A}-\hat{\theta}_{k}I) \hat{\mathbf{u}}_{k}$ be the residual at the $k^{\mathrm{th}}$ step of the Rayleigh Quotient Iteration. If $\hat{A}$ is Hermitian, then the sequence $\left\{\left\|\hat{\mathbf{r}}_{k}\right\|_{2}, k=0,1, \cdots\right\}$ is monotonically decreasing for any initial vector $\hat{\mathbf{u}}_{0}$. 
\end{theorem}

\begin{proof}
	\begin{align}
		\left\|\hat{\mathbf{r}}_{k+1}\right\|_{2} & =\|(\hat{A}-\hat{\theta}_{k+1}I) \hat{\mathbf{u}}_{k+1}\|_{2} \nonumber \\
		& \leqq\|(\hat{A}-\hat{\theta}_{k}I) \hat{\mathbf{u}}_{k+1}\|_{2} \label{residual1} \\
		& =|\hat{\mathbf{u}}_{k}^{\ast}(\hat{A}-\hat{\theta}_{k}I) \hat{\mathbf{u}}_{k+1}| \label{residual2}  \\
		& \leqq\|\hat{\mathbf{u}}_{k}^{\ast}(\hat{A}-\hat{\theta}_{k}I)\|_{2}\|\hat{\mathbf{u}}_{k+1}\|_{2} \label{residual3}  \\
		& =\|\hat{\mathbf{u}}_{k}^{\ast}(\hat{A}-\hat{\theta}_{k}I)\|_{2} \nonumber  \\
		& =\|(\hat{A}-\hat{\theta}_{k}I) \hat{\mathbf{u}}_{k}\|_{2} \nonumber  \\
		& =\|\hat{\mathbf{r}}_{k}\|_{2}. \nonumber 
	\end{align}
	Inequality \eqref{residual1} holds by Lemma \ref{mini}. From lines 3 and 4 of Algorithm \ref{RQI}, we have
	\begin{equation}\label{lines34}
		\begin{cases}
			(\hat{A}-\hat{\theta}_{k}I)\hat{\mathbf{w}}_{k+1}=\hat{\mathbf{u}}_{k} \\ 
			\hat{\mathbf{u}}_{k+1}=\hat{\mathbf{w}}_{k+1} / \|\hat{\mathbf{w}}_{k+1}\|_{2}
		\end{cases} \Longrightarrow (\hat{A}-\hat{\theta}_{k}I)\hat{\mathbf{u}}_{k+1}=\dfrac{\hat{\mathbf{u}}_{k}}{\|\hat{\mathbf{w}}_{k+1}\|_{2}}
	\end{equation}
	Taking the 2-norm on both sides of \eqref{lines34}, we obtain 
	\begin{equation*}
		\left\|(\hat{A}-\hat{\theta}_{k}I) \hat{\mathbf{u}}_{k+1}\right\|_{2} =
		\left|\hat{\mathbf{u}}_{k}^{\ast}(\hat{A}-\hat{\theta}_{k}I) \hat{\mathbf{u}}_{k+1}\right|,
	\end{equation*}
	then \eqref{residual2} holds. Inequality \eqref{residual3} holds due to the Cauchy-Schwarz inequality (Proposition 5.5 in \cite{qi2021eigenvalues}).
	Equality for
	\begin{equation*}
		\left\|\hat{\mathbf{r}}_{k+1}\right\|_{2} \leq \left\|\hat{\mathbf{r}}_{k}\right\|_{2}
	\end{equation*}
	is achieved when $\hat{\theta}_{k+1}=\hat{\theta}_{k}$ and $\hat{\mathbf{u}}_{k+1}^{\ast}$ is parallel to $\hat{\mathbf{u}}_{k}^{\ast}(\hat{A}-\hat{\theta}_{k}I)$.
\end{proof}

\begin{theorem}\label{main}
	Given a dual quaternion Hermitian matrix $\hat{A}\in \mathbb{DQ}^{n\times n}$. Let the RQI (Algorithm \ref{RQI}) be applied to a dual quaternion Hermitian matrix $\hat{A}$ with a proper starting vector $\hat{\mathbf{u}}_{0}$. When $k$ goes to infinity,
	
	\noindent $(1)$ $(\hat{\theta}_{k}, \hat{\mathbf{u}}_{k})$ converges to an eigenpair $(\lambda, \hat{\mathbf{x}})$ of $\hat{A}$, or
	
	\noindent $(2)$ $\hat{\theta}_{k}$ converges to a point equidistant from an eigenvalue of $\hat{A}$, and the sequence  $\left\{\hat{\mathbf{u}}_{k}\right\}$ cannot converge.
\end{theorem}

\begin{proof}
	By Theorem \ref{mono}, the monotone sequence $\left\{\left\|\hat{\mathbf{r}}_{k}\right\|_{2}\right\}$ is bounded below by $0$. Let
	\begin{equation}\label{lim1}
		\lim \limits_{k \rightarrow \infty}\left\|\hat{\mathbf{r}}_{k}\right\|_{2}=\lim \limits_{k \rightarrow \infty}\|(\hat{A}-\hat{\theta_{k}}) \hat{\mathbf{u}}_{k}\|_{2} \rightarrow \tau \geq 0.
	\end{equation}	
	Since the sequence  $\left\{\mathbf{u}_{k}\right\}_{k=0}^{\infty}$ is confined to the unit sphere, a compact subset of  $\mathbb{DQ}^{n}$, $\left\{\mathbf{u}_{k}\right\}_{k=0}^{\infty}$ must have at least one accumulation point. Additionally, the $\{\hat{\theta}_{k}\}_{k=0}^{\infty}$ are confined to the numerical range of $\hat{A}$, which is a closed and bounded region over dual numbers. Therefore, the sequence $\{\hat{\theta}_{k}\}_{k=0}^{\infty}$ have limit points too. 
	
	We discuss the following two cases separately based on the limit of the residual.
	
	Case 1: $\tau=0$ (the usual case). 
	By definition, if $(\hat{\theta}, \hat{\mathbf{u}})$ is an accumulation point of the sequence $\{\hat{\theta}_{k}, \hat{\mathbf{u}}_{k}\}_{k=0}^{\infty}$, then there exists a subsequence ${(\hat{\theta}_{i}, \hat{\mathbf{u}}_{i})}$ indexed by some set $\mathcal{M}$, which converges to $(\hat{\theta}, \hat{\mathbf{u}})$. When $i$ goes to infinity in $\mathcal{M}$, because $\hat{\theta}$ is a continuous function on the unit sphere, then we have
	\begin{equation*}
		\hat{\theta}(\hat{\mathbf{u}})=\lim _{i\in \mathcal{M}} \hat{\theta}(\hat{\mathbf{u}}_{i})=\hat{\theta}
	\end{equation*}
	and taking limits in this subsequence yields
	\begin{equation*}
		\|(\hat{A}-\hat{\theta}I) \hat{\mathbf{u}}\|_{2}=\lim _{i\in \mathcal{M}}\|(\hat{A}-\hat{\theta}_{i}I) \hat{\mathbf{u}}_{i}\|_{2}=\tau=0.
	\end{equation*}
	Thus, $(\hat{\theta},\hat{\mathbf{u}})$ must be an eigenpair of $A$. 
	
	Case 2. $\tau>0$. Note that
	\begin{align}
		\|\hat{\mathbf{r}}_{k+1}\|_{2}^{2} & =\|(\hat{A}-\hat{\theta}_{k+1}I) \hat{\mathbf{u}}_{k+1}\|_{2}^{2} \nonumber\\
		& =\|\hat{A}\hat{\mathbf{u}}_{k+1}-\hat{\theta}_{k+1}\hat{\mathbf{u}}_{k+1}-\hat{\theta}_{k}\hat{\mathbf{u}}_{k+1}+\hat{\theta}_{k}\hat{\mathbf{u}}_{k+1}\|_{2}^{2} \nonumber \\
		& = \|(\hat{A}-\hat{\theta}_{k}I) \hat{\mathbf{u}}_{k+1}+(\hat{\theta}_{k}-\hat{\theta}_{k+1})\hat{\mathbf{u}}_{k+1}\|_{2}^{2} \nonumber \\
		& =\|(\hat{A}-\hat{\theta}_{k}I) \hat{\mathbf{u}}_{k+1}\|_{2}^{2}+|\hat{\theta}_{k}-\hat{\theta}_{k+1}|^{2}+2|\hat{\theta}_{k}-\hat{\theta}_{k+1}|\hat{\mathbf{u}}_{k+1}^{\ast}(\hat{A}-\hat{\theta}_{k}I)\hat{\mathbf{u}}_{k+1} \label{normex} \\
		& = \|(\hat{A}-\hat{\theta}_{k}I) \hat{\mathbf{u}}_{k+1}\|_{2}^{2}-|\hat{\theta}_{k}-\hat{\theta}_{k+1}|^{2} \nonumber \\
		& = |\hat{\mathbf{u}}_{k}^{\ast}(\hat{A}-\hat{\theta}_{k}I) \hat{\mathbf{u}}_{k+1}|^{2}-|\hat{\theta}_{k}-\hat{\theta}_{k+1}|^{2} \nonumber \\
		& =\|(\hat{A}-\hat{\theta}_{k}I)\hat{\mathbf{u}}_{k}\|_{2}^{2}\cdot\kappa_{k}^{2}-|\hat{\theta}_{k}-\hat{\theta}_{k+1}|^{2} \label{kap} \\
		&=\|\hat{\mathbf{r}}_{k}\|_{2}^{2}\cdot\kappa_{k}^{2}-|\hat{\theta}_{k}-\hat{\theta}_{k+1}|^{2}, \nonumber
	\end{align}
	Equality \eqref{normex} is obtained by Lemma \ref{normsum}. Let
	\begin{equation*}
		\langle \hat{\mathbf{u}}_{k+1},(\hat{A}-\hat{\theta}_{k}I)\hat{\mathbf{u}}_{k}\rangle=\|\hat{\mathbf{u}}_{k+1}\|_{2}\|(\hat{A}-\hat{\theta}_{k}I)\hat{\mathbf{u}}_{k}\|_{2}\cdot\kappa_{k}=\|(\hat{A}-\hat{\theta}_{k}I)\hat{\mathbf{u}}_{k}\|_{2}\cdot\kappa_{k}
	\end{equation*}
	in \eqref{kap}. By the Cauchy-Schwarz inequality, we know that $\kappa_{k}^{2} \leq 1$.
	
	Since 
	\begin{equation*}
		\lim \limits_{k \rightarrow \infty}\left\|\hat{\mathbf{r}}_{k}\right\|_{2}=\tau
	\end{equation*}
	and the monotonic residual property must hold in the limit, then we have
	\begin{equation*}
		\lim \limits_{k \rightarrow \infty}|\hat{\theta}_{k}-\hat{\theta}_{k+1}| = 0.
	\end{equation*}
	Moreover, the sequence $\{\hat{\theta}_{k}\}_{k=0}^{\infty}$ have limit points, there can only be one of them. Thus,
	\begin{equation*}
		\hat{\theta}_{k} \rightarrow \theta, \ \text{as} \ k \rightarrow \infty.
	\end{equation*}
	
	As $k \rightarrow \infty$,
	\begin{equation*}
		\kappa_{k}^{2}=\left(\left\|\mathbf{r}_{k+1}\right\|_{2}^{2}+\left|\hat{\theta}_{k}-\hat{\theta}_{k+1}\right|^{2}\right) / \left\|\hat{\mathbf{r}}_{k}\right\|_{2}^{2} \rightarrow 1,
	\end{equation*}
	so we have
	\begin{equation}\label{lim2}
		\tau_{k}=\left|\hat{\mathbf{u}}_{k}^{\ast}\left(\hat{A}-\hat{\theta}_{k}I\right) \hat{\mathbf{u}}_{k+1}\right|=\left\|\hat{\mathbf{r}}_{k}\right\|_{2}\cdot \kappa_{k} \rightarrow \tau.
	\end{equation}
	Now
	\begin{equation*}
		\begin{aligned}
			\tau_{k}^{2} & =\hat{\mathbf{u}}_{k+1}^{\ast}\left(\hat{A}-\hat{\theta}_{k}I\right)^{\ast}\left(\hat{A}-\hat{\theta}_{k}I\right) \hat{\mathbf{u}}_{k+1} \\
			& \leqq\left\|\hat{\mathbf{u}}_{k+1}^{\ast}\right\|_{2}\left\|\left(\hat{A}-\hat{\theta}_{k}I\right)^{\ast} \left(\hat{A}-\hat{\theta}_{k}I\right) \hat{\mathbf{u}}_{k+1}\right\|_{2}, \\
			& =\left\|\left(\hat{A}-\hat{\theta}_{k}I\right)^{\ast} \hat{\mathbf{u}}_{k} \tau_{k}\right\|_{2}, \quad \text { since }\left(\hat{A}-\hat{\theta}_{k}I\right) \hat{\mathbf{u}}_{k+1}=\hat{\mathbf{u}}_{k} \tau_{k}, \\
			& =\left\|\hat{\mathbf{r}}_{k}\right\|_{2}\cdot\tau_{k}.
		\end{aligned}
	\end{equation*}
	Combining \eqref{lim1} and \eqref{lim2}, equality in the Cauchy-Schwarz inequality must hold in the limit, as $k \rightarrow \infty$. Let $\mathcal{N}$ be a subsequence of $\{0,1,2, \cdots\}$ such that 
	$$
	\lim \limits_{k \in \mathcal{N}} \hat{\mathbf{u}}_{k}=\hat{\mathbf{u}}.$$
	Since $\|(\hat{A}-\hat{\theta}I)\hat{\mathbf{u}}\|_{2}=\tau$, we have $(\hat{A}-\hat{\theta}I)^{\ast}(\hat{A}-\hat{\theta}I)\hat{\mathbf{u}}=\hat{\mathbf{u}}\tau^{2}$ for each limit point. Using the same subsequence again, we find
	\begin{equation*}
		\hat{\theta}=\lim_{k \in \mathcal{N}} \hat{\theta}\left(\hat{\mathbf{u}}_{k}\right)=\hat{\theta}\left(\hat{\mathbf{u}}\right).
	\end{equation*}
	This shows that each limit point $\hat{\mathbf{u}}$ of  $\left\{\mathbf{u}_{k}\right\}_{k=0}^{\infty}$ is a singular vector of $\hat{A}-\hat{\theta}I$ with associated singular value $\tau$. Note that $\hat{A}-\hat{\theta}I$ is Hermitian, thus 
	\begin{equation*}
		\tau=|\lambda_{i}-\hat{\theta}|,
	\end{equation*}
	where $\lambda_{i}$ is one of the eigenvalues of $\hat{A}$. This completes the proof.
\end{proof}

\subsection{Computing All Appreciable Eigenvalues of a Dual Quaternion Hermitian Matrix}

\noindent We want to obtain the remaining eigenvalues and their corresponding eigenvectors. According to Theorem 7.1 in \cite{ling2022singular}, the Eckart-Young-like theorem is applicable to dual quaternion matrices. Specifically, if we have a dual quaternion Hermitian matrix, denoted as $\hat{A} \in \mathbb{DQ}^{n \times n}$, then, as indicated in Theorem \ref{unitary}, $\hat{A}$ can be expressed in an alternate form as follows
\begin{equation}\label{eqall}
\hat{A}=\hat{U} \Sigma \hat{U}^{\ast}=\sum_{i=1}^{n} \lambda_{i} \hat{\mathbf{u}}_{i} \hat{\mathbf{u}}_{i}^{*},
\end{equation}
where  
$\hat{U}=\left[\hat{\mathbf{u}}_{1}, \ldots, \hat{\mathbf{u}}_{n}\right] \in \mathbb{DQ}^{n \times n}$ is a unitary matrix,  $\Sigma=\operatorname{diag}\left(\lambda_{1}, \ldots, \lambda_{n}\right) \in \mathbb{D}^{n \times n}$ is a diagonal dual matrix, and $\lambda_{i}, i=1, \ldots, n$ are in the descending order.

Denote  $\hat{A}_{k}=\hat{A}-\sum_{i=1}^{k-1} \lambda_{i} \hat{\mathbf{u}}_{i} \hat{\mathbf{u}}_{i}^{*}$. It follows that $\lambda_{k}$ and $\hat{\mathbf{u}}_{k}$ constitute an eigenpair of $\hat{A}_{k}$. When $\lambda_{k} \neq 0, \lambda_{k}$ and $\hat{\mathbf{u}}_{k}$ both can be computed by applying the Rayleigh quotient iteration to $\hat{A}_{k}$. By iteratively repeating this process for $k=1$ to $n$, we can ascertain all appreciable eigenvalues and the associated eigenvectors. A concise summary of this procedure can be found in Algorithm \ref{RQI_all}.

\begin{algorithm}[H]
\caption{\ Computing all appreciable eigenvalues.}
\begin{algorithmic}[1]\label{RQI_all}
	\REQUIRE Given a dual quaternion Hermitian matrix $\hat{A} \in \mathbb{DQ}^{n \times n}$, the tolerance $\gamma$.
	\ENSURE Eigenvalues $\{\lambda_{i}\}_{i=1}^{k}$ and eigenvectors $\{\hat{\mathbf{u}}_{i}\}_{i=1}^{k}$.
	\STATE Set $\hat{A}_{1}=\hat{A}$.
	\FOR {$k=1,2,...,$}
	\STATE Compute $\lambda_{k},\hat{\mathbf{u}}_{k}$ as the eigenpair of $\hat{A}_{k}$ by Algorithm \ref{RQI}.
	\STATE Update $\hat{A}_{k+1}=\hat{A}_{k}-\lambda_{k}\hat{\mathbf{u}}_{k}\hat{\mathbf{u}}_{k}^{\ast}$.
	\STATE If $\|\tilde{A}_{k+1,\mathrm{st}}\|_{F} \leqslant \gamma$, then stop.
	\ENDFOR
\end{algorithmic}
\end{algorithm}

Finally, we give the convergence theorem for Algorithm \ref{RQI_all}.

\begin{theorem}
Given a dual quaternion Hermitian matrix $A \in \mathbb{DQ}^{n \times n}$. Then Algorithm $\ref{RQI_all}$ can return all appreciable eigenvalues and their corresponding eigenvectors of $A$.
\end{theorem} 

\begin{proof}
The theorem can be derived from the combination of equation (\ref{eqall}) and Theorem \ref{main}.
\end{proof} 

\section{Numerical Experiments}\label{section5}

\noindent In this section, some numerical examples are used to illustrate that Algorithm \ref{RQI_all} is feasible and effective to solve the dual quaternion Hermitian eigenvalue problem (\ref{problem}). All the tests are performed under Windows 11 and MATLAB version 23.2.0.2365128 (R2023b) with an AMD Ryzen 7 5800H with Radeon Graphics CPU at 3.20 GHz and 16 GB of memory.

\vskip 2mm

\noindent \textbf{Example 1} ~ Consider Problem (\ref{problem}) with a random dual quaternion Hermitian matrix $\hat{A}=\tilde{A}_{\mathrm{st}}+\tilde{A}_{\mathcal{I}}\epsilon \in \mathbb{DQ}^{6\times 6}$, where the standard part of $\hat{A}$ is $\tilde{A}_{\mathrm{st}}=A_{\mathrm{st},1}+A_{\mathrm{st},2}\mathbf{i}+A_{\mathrm{st},3}\mathbf{j}+A_{\mathrm{st},4}\mathbf{k}\in \mathbb{Q}^{6\times 6}$, $A_{\mathrm{st},1},A_{\mathrm{st},2},A_{\mathrm{st},3}$ and $A_{\mathrm{st},4}$ are equal to
\begin{equation*}
\resizebox{\linewidth}{!}{$
	\begin{bmatrix}
		0.6634 & 0.1840 & -0.0967 & -0.0590 & 0.1392 & -0.0366\\
		0.1840 & 0.6872 & 0.7003 & 0.2440 & -0.0937 & 0.2263\\
		-0.0967 & 0.7003 & 0.8900 & 0.3074 & -0.1671 & 0.1731\\
		-0.0590 & 0.2440 & 0.3074 & 0.1619 & -0.0953 & 0.0129\\
		0.1392 & -0.0937 & -0.1671 & -0.0953 & 0.0738 & -0.0184\\
		-0.0366 & 0.2263 & 0.1731 & 0.0129 & -0.0184 & 0.46621
	\end{bmatrix}, \begin{bmatrix}
		0 & -0.3887 & -0.4726 & -0.0169 & -0.0043 & -0.3205\\
		0.3887 & 0 & -0.3203 & -0.0722 & 0.0841 & 0.2418\\
		0.4726 & 0.3203 & 0 & 0.0054 & 0.0602 & 0.4245\\
		0.0169 & 0.0722 & -0.0054 & 0 & -0.0046 & 0.2474\\
		0.0043 & -0.0841 & -0.0602 & 0.0046 & 0 & -0.1842\\
		0.3205 & -0.2418 & -0.4245 & -0.2474 & 0.1842 & 0
	\end{bmatrix},
	$}
\end{equation*}
\begin{equation*}
\resizebox{\linewidth}{!}{$
	\begin{bmatrix}
		0 & 0.5196 & 0.5832 & 0.3211 & -0.1718 & 0.0341\\
		-0.5196 & 0 & 0.1365 & 0.1215 & -0.1603 & 0.2871\\
		-0.5832 & -0.1365 & 0 & -0.0072 & -0.1006 & 0.4132\\
		-0.3211 & -0.1215 & 0.0072 & 0 & -0.0521 & 0.0369\\
		0.1718 & 0.1603 & 0.1006 & 0.0521 & 0 & 0.0034\\
		-0.0341 & -0.2871 & -0.4132 & -0.0369 & -0.0034 & 0
	\end{bmatrix}, \begin{bmatrix}
		0 & -0.0324 & 0.1329 & 0.0221 & 0.0059 & -0.4517\\
		0.0324 & 0 & 0.0037 & 0.1781 & -0.0956 & -0.3582\\
		-0.1329 & -0.0037 & 0 & 0.2225 & -0.1549 & -0.1844\\
		-0.0221 & -0.1781 & -0.2225 & 0 & -0.0116 & -0.1129\\
		-0.0059 & 0.0956 & 0.1549 & 0.0116 & 0 & -0.0112\\
		0.4517 & 0.3582 & 0.1844 & 0.1129 & 0.0112 & 0
	\end{bmatrix},
	$}
\end{equation*}
respectively. The dual part of $\hat{A}$ is $\tilde{A}_{\mathcal{I}}=A_{\mathcal{I},1}+A_{\mathcal{I},2}\mathbf{i}+A_{\mathcal{I},3}\mathbf{j}+A_{\mathcal{I},4}\mathbf{k}\in \mathbb{Q}^{6\times 6}$, $A_{\mathcal{I},1},A_{\mathcal{I},2},A_{\mathcal{I},3}$ and $A_{\mathcal{I},4}$ are equal to
\begin{equation*}
\resizebox{\linewidth}{!}{$
	\begin{bmatrix}
		-0.8130 & -0.0039 & 0.2942 & -0.1255 & -0.1397 & -0.0313\\
		-0.0039 & -0.3840 & -0.2356 & 0.1480 & 0.8592 & 0.4817\\
		0.2942 & -0.2356 & -0.0863 & 0.4712 & 1.0305 & 0.2197\\
		-0.1255 & 0.1480 & 0.4712 & 0.2016 & 0.3178 & 0.6339\\
		-0.1397 & 0.8592 & 1.0305 & 0.3178 & -0.3965 & -0.3540\\
		-0.0313 & 0.4817 & 0.2197 & 0.6339 & -0.3540 & 0.2849
	\end{bmatrix}, \begin{bmatrix}
		0 & 0.0689 & 0.4862 & -0.4231 & -0.3142 & 0.1983\\
		-0.0689 & 0 & 0.0069 & -0.0679 & -0.5687 & -0.2706\\
		-0.4862 & -0.0069 & 0 & 0.2682 & -0.4452 & 0.1949\\
		0.4231 & 0.0679 & -0.2682 & 0 & -0.0837 & 0.2557\\
		0.3142 & 0.5687 & 0.4452 & 0.0837 & 0 & 0.5107\\
		-0.1983 & 0.2706 & -0.1949 & -0.2557 & -0.5107 & 0
	\end{bmatrix},
	$}
\end{equation*}
\begin{equation*}
\resizebox{\linewidth}{!}{$
	\begin{bmatrix}
		0 & -0.7266 & -0.2631 & -0.0285 & 0.8539 & 0.5629\\
		0.7266 & 0 & -0.2641 & 0.2712 & -0.0399 & 0.4132\\
		0.2631 & 0.2641 & 0 & 0.4700 & -0.4774 & 0.0773\\
		0.0285 & -0.2712 & -0.4700 & 0 & -0.0426 & 0.1856\\
		-0.8539 & 0.0399 & 0.4774 & 0.0426 & 0 & 0.6367\\
		-0.5629 & -0.4132 & -0.0773 & -0.1856 & -0.6367 & 0
	\end{bmatrix}, \begin{bmatrix}
		0 & 0.0525 & 0.1599 & -0.1970 & 0.5612 & 0.1146\\
		-0.0525 & 0 & 0.7538 & -0.2008 & 0.2974 & 0.4295\\
		-0.1599 & -0.7538 & 0 & -0.2703 & 0.1851 & 0.2496\\
		0.1970 & 0.2008 & 0.2703 & 0 & -0.2590 & 0.0729\\
		-0.5612 & -0.2974 & -0.1851 & 0.2590 & 0 & -0.3074\\
		-0.1146 & -0.4295 & -0.2496 & -0.0729 & 0.3074 & 0
	\end{bmatrix},
	$}
\end{equation*}
respectively.

In this experiment, we apply the Rayleigh quotient iteration (Algorithm \ref{RQI_all}) to compute all appreciable eigenvalues of $\hat{A}$ with an initial unit dual quaternion vector
\begin{equation*}
	\resizebox{\linewidth}{!}{$
		\hat{\mathbf{q}}=\begin{bmatrix}
			-0.5514-1.3693\mathbf{i}-0.6843\mathbf{j}-0.5685\mathbf{k}\\
			0.2849+0.6141\mathbf{i}+2.7228\mathbf{j}+0.9304\mathbf{k} \\
			-1.2502+0.1685\mathbf{i}-0.9310\mathbf{j}+0.4428\mathbf{k} \\
			-0.4044-1.0874\mathbf{i}+0.1983\mathbf{j}+1.3718\mathbf{k} \\
			-2.4265-0.4707\mathbf{i}-0.1511\mathbf{j}-0.4807\mathbf{k} \\
			0.8045+0.0331\mathbf{i}+0.8616\mathbf{j}+0.4602\mathbf{k}
		\end{bmatrix}+\begin{bmatrix}
			1.2034-0.3999\mathbf{i}+1.5179\mathbf{j}+0.7439\mathbf{k} \\
			0.4025-1.5399\mathbf{i}+1.1467\mathbf{j}-0.5636\mathbf{k} \\
			1.2182-0.8999\mathbf{i}-0.1223\mathbf{j}+0.8272\mathbf{k} \\
			-0.0179-0.3404\mathbf{i}-0.2029\mathbf{j}+1.3261\mathbf{k} \\
			-0.9534+0.6708\mathbf{i}+0.2699\mathbf{j}+0.8289\mathbf{k} \\
			-0.2012+0.0505\mathbf{i}+0.5192\mathbf{j}+0.1988\mathbf{k}
		\end{bmatrix}\epsilon.
		$}
\end{equation*}
This is a toy example to show the validity and reliability of our algorithm. We use the residual $\|\hat{A} \hat{\mathbf{v}}_{k}- \hat{\mathbf{v}}_{k} \lambda_{k}\|_{2^R}\leq 10^{-5}$ as the stopping criterion. The RQI algorithm successively obtained the following six eigenvalues of the matrix $\hat{A}$, and they are all dual numbers.
\begin{equation*}
	\begin{aligned}
		& 58.248-14.5130\epsilon, & 35.691+4.1262\epsilon,~ & 21.769+8.1369\epsilon, \\
		& 11.176-5.9870\epsilon, & 6.8844-2.0823\epsilon,~ & 2.9425-1.1933\epsilon.
	\end{aligned}
\end{equation*}
Meanwhile, the RQI algorithm obtains the residuals $\|\hat{A} \hat{\mathbf{v}}_{i}- \hat{\mathbf{v}}_{i} \lambda_{i}\|_{2^R}$ for the above $6$ eigenvalues are 
\begin{equation*}
	\begin{aligned}
		& 1.170094539721946e-06, ~ 2.179839717407519e-09, \\
		& 9.487484804907692e-08, ~ 1.442736183487963e-07, \\
		& 2.537630772794933e-06, ~ 6.023193804097548e-07,
	\end{aligned}
\end{equation*}
respectively. Throughout the process, the RQI algorithm iterates a total of $26$ steps and takes $0.0348$ seconds. In addition, we compare it with the power method (PM) \cite{cui2023power}, and the results are shown in Table \ref{tab1}. The convergence curves of the Rayleigh quotient iteration are given in Figure \ref{fig.1}. When computing the eigenvalues of the dense dual quaternion Hermitian matrix $\hat{A}$, the efficiency of RQI surpasses that of PM in terms of iteration steps and CPU running time. Example 1 shows that Algorithm \ref{RQI} is feasible and effective to solve Problem \ref{problem}.

\begin{table}[H]
	\centering
	\caption{The average numerical results of the RQI and PM methods for computing the $6$ eigenvalues of the dual quaternion dense matrix $A$.
	}
	\label{tab1}
		\begin{tabular}{ccccc}
			\toprule
			Method& Residual $\|\hat{A} \hat{\mathbf{v}}- \hat{\mathbf{v}} \lambda\|_{2^R}$ & Relative Error $|\hat{\lambda}-\lambda_{i}|/|\lambda|$ & IT & CPU \\
			\hline
			PM & 6.143314328534241e-06 & 3.132564162153321e-06 & 24.3333 & 0.0244 \\
			RQI & 7.585620775193183e-07 & 3.127950914619532e-06 & 4.3333 & 0.0058 \\
			\bottomrule
		\end{tabular}
\end{table}

\begin{figure}[htbp]
	\centering
	\includegraphics[scale=0.57]{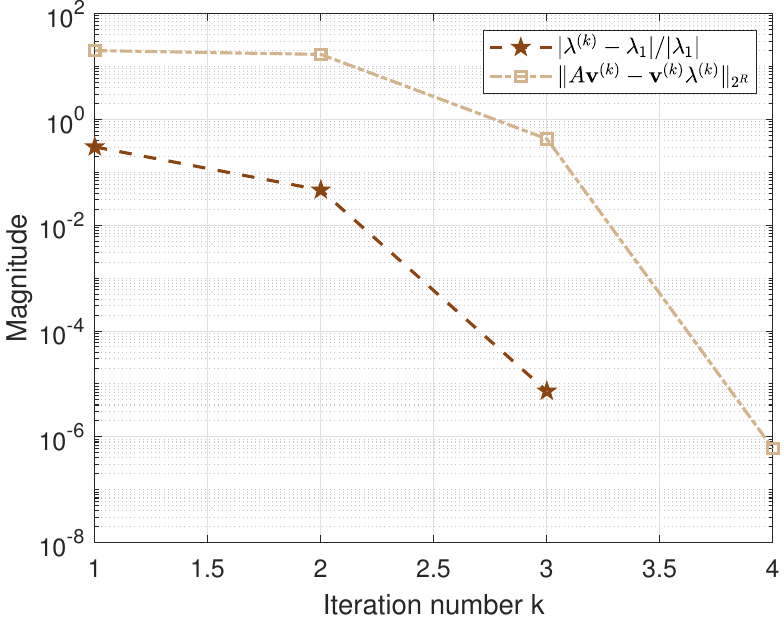}\quad
	\includegraphics[scale=0.57]{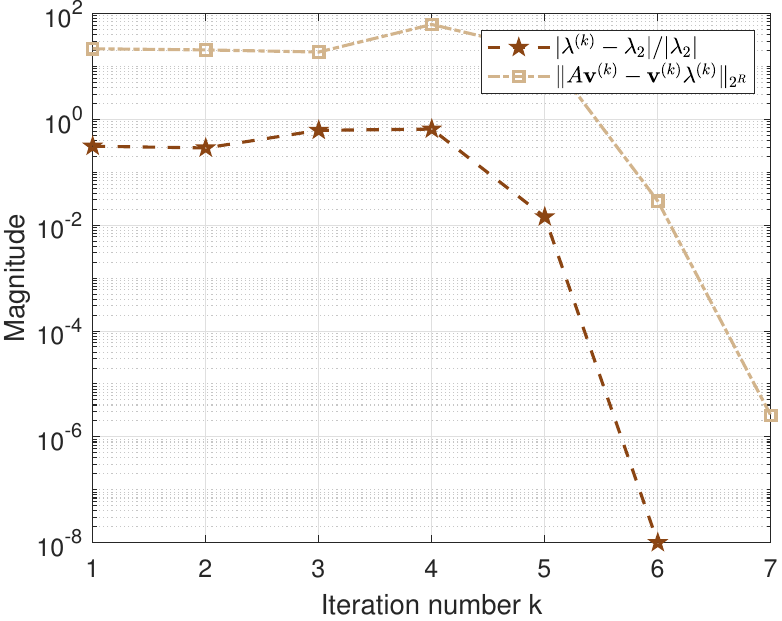}
	\\
	\includegraphics[scale=0.57]{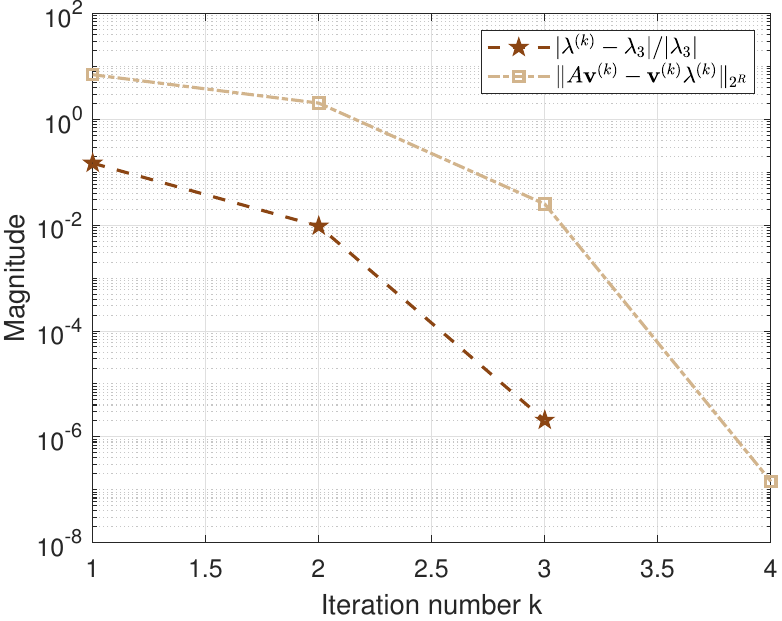}\quad
	\includegraphics[scale=0.57]{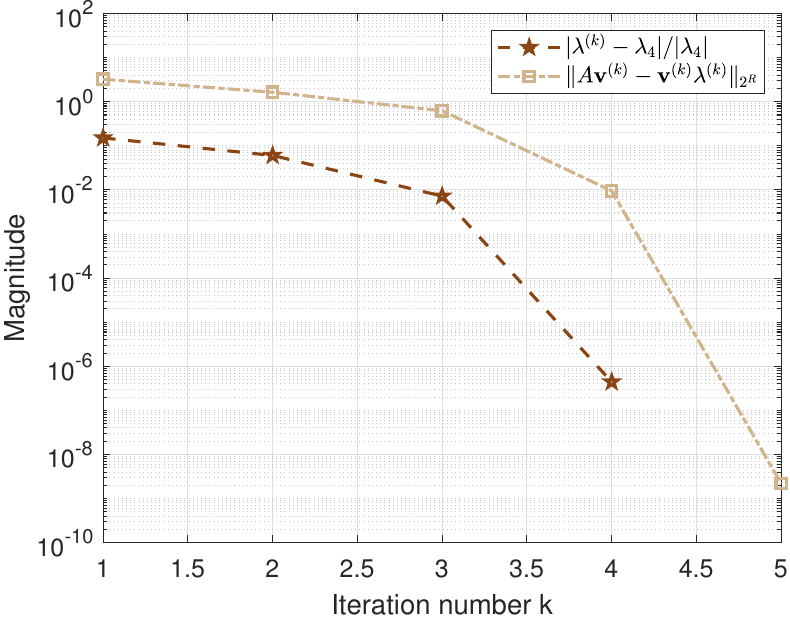}\\
	\includegraphics[scale=0.57]{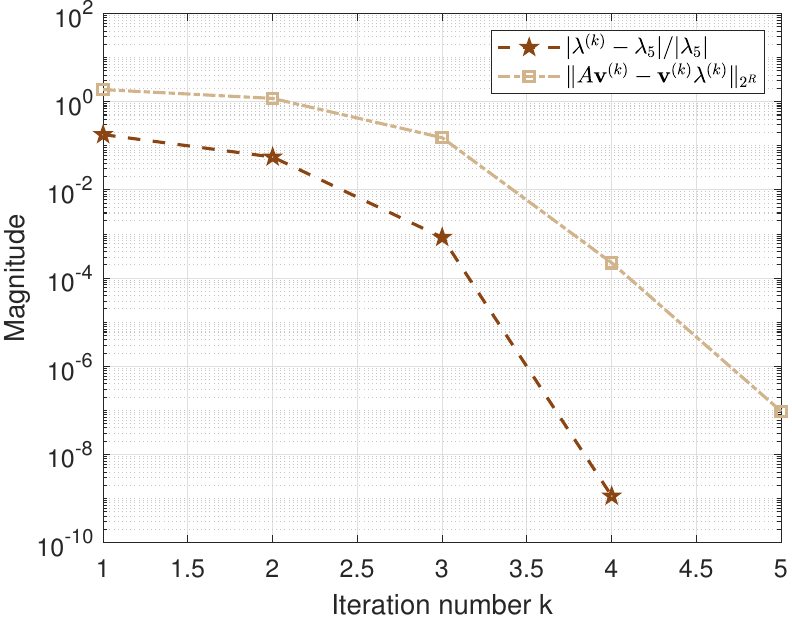}\quad
	\includegraphics[scale=0.57]{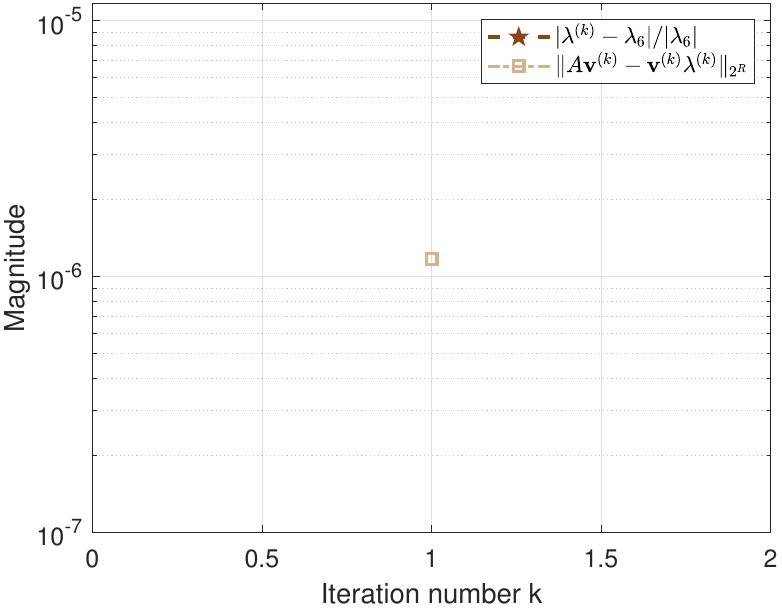}
	\caption{The convergence curves of RQI for computing all appreciable eigenvalues of the dual quaternion Hermitian matrix $\hat{A}$.}
	\label{fig.1}
\end{figure}

\vskip 2mm

\noindent \textbf{Example 2} ~ Consider Problem (\ref{problem}) with the random Laplacian matrices. In the multi-agent formation control, as detailed in \cite{qi2023dualcontrol,lin2014distributed}, the interactions among the rigid bodies are typically represented by an undirected graph. We make the assumption that these rigid bodies possess omni-directional capabilities, meaning that the ability of rigid body $i$ to sense rigid body $j$ is reciprocal. Consequently, a mutual visibility graph, denoted as $G=(V,E)$, is employed to model the collective behavior of multiple rigid bodies. The vertex set $V$ represents the individual rigid bodies, and an edge $(i, j) \in E$ is present if and only if rigid bodies $i$ and $j$ share mutual visual perception.

For a given graph $G=(V, E)$ with $n$ vertices and $m$ edges, the Laplacian matrix of $G$ can be defined as follows
$$
L=D-\hat{A},
$$
where $D$ is a diagonal real matrix, with diagonal entries signifying the degrees of the respective vertices, and $\hat{A}=\left(\hat{a}_{i j}\right)$ is
$$
\hat{a}_{i j}=\begin{cases}
	\hat{\mathbf{q}}_{i}^{\ast} \hat{\mathbf{q}}_{j}, & \mbox{if} ~ (i,j) \in E, \\
	\hat{0}, & \mbox{otherwise},
\end{cases}
$$
where $\hat{\mathbf{q}}=(\hat{q}_{i})\in \mathbb{DQ}^{n \times 1}$ with $\|\hat{\mathbf{q}}\|_{2}=1$. In the domain of multi-agent formation control, $\hat{q}_{i}$ represents the configuration of the $i^{\mathrm{th}}$ rigid body, while $a_{ij}$ signifies the relative configuration between the rigid bodies $i$ and $j$. Additionally, the relative configuration adjacency matrix is denoted as $\hat{A}$. To ensure that all eigenvalues are appreciable, we introduce a small perturbation to the Laplacian matrix. Let $\check{L}=L+\alpha I_{n}$, where $\alpha$ is a nonzero real number. Then $\lambda$ is an eigenvalue of $L$ if and only if  $\lambda+\alpha$ is an eigenvalue of $\check{L}$.

We define the residual 
\begin{equation*}
	\|\hat{L} \hat{\mathbf{v}}_{k}- \hat{\mathbf{v}}_{k} \lambda_{k}\|_{2^R}
\end{equation*}
of Problem (\ref{problem}) and the relative error 
\begin{equation*}
	\dfrac{|\lambda_{k}-\lambda_{i}|}{|\lambda_{i}|}
\end{equation*}
of the approximation to the desired eigenvalues during iteration. The two indices  are used to check whether the implemented methods converge to the desired eigenvalue with the associated eigenvector.

\begin{table}[H]
	\centering
	\caption{Numerical results of RQI and PM methods for computing an eigenvalue of the random Laplacian matrices, where $\mathrm{RSE}:=\|\hat{L} \hat{\mathbf{v}}- \hat{\mathbf{v}} \lambda\|_{2^R}$.
	}
	\label{tab2}
	\resizebox{\linewidth}{!}{
		\begin{tabular}{c|c|cccccc}
			\toprule
			Method&$n$& $10$ & $20$  & $50$ & $100$ & $200$ & $400$ \\
			\hline
			&IT&4  & 5 & 4 & 4 &4&7\\
			RQI&CPU& 0.0056  & 0.0116 & 0.0236 & 0.0783 &0.3915&4.6050\\
			&RSE&1.0710e-6  & 8.7114e-6 & 8.9354e-9 & 3.1376e-6 &5.6876e-12&1.8686e-9\\
			\hline
			&IT&104  & 408 & 1940 & 6925 & - &-\\
			PM &CPU&0.0653  & 0.2462 & 1.1549 & 4.6902 &-&-\\
			&RSE&1.2649e-5  & 1.3372e-5 & 2.3935e-5 & 3.6778e-5 &-&-\\
			\bottomrule
		\end{tabular}
	}
\end{table}

Then, we further consider the Laplacian matrix of circles with $10,20,50,100$ points. The residual of an eigenpair $(\lambda_{k}, \hat{\mathbf{v}}_{k})$ is measured by $\|\hat{L} \hat{\mathbf{v}}_{k}- \hat{\mathbf{v}}_{k} \lambda_{k}\|_{2^R}$. We use
either the residual $\|\hat{L} \hat{\mathbf{v}}_{k}- \hat{\mathbf{v}}_{k} \lambda_{k}\|_{2^R}\leq 10^{-5}$ or the number of iteration step has reached the upper limit of 15000 as the stopping criterion.

\begin{figure}[h]
	\centering
	\subfigure[10-point circle, $L\in \mathbb{DQ}^{10\times 10}$]{\includegraphics[scale=0.57]{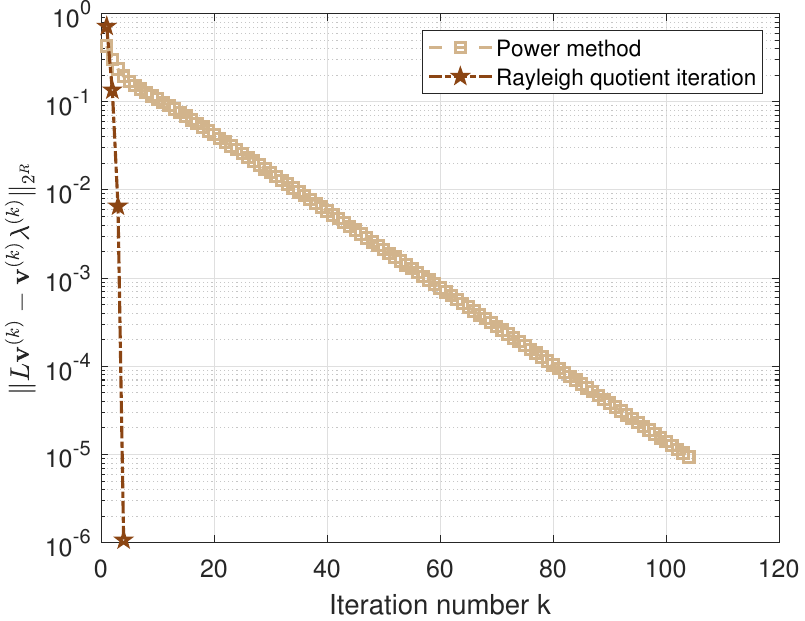}} 
	\subfigure[20-point circle, $L\in \mathbb{DQ}^{20\times 20}$]{\includegraphics[scale=0.57]{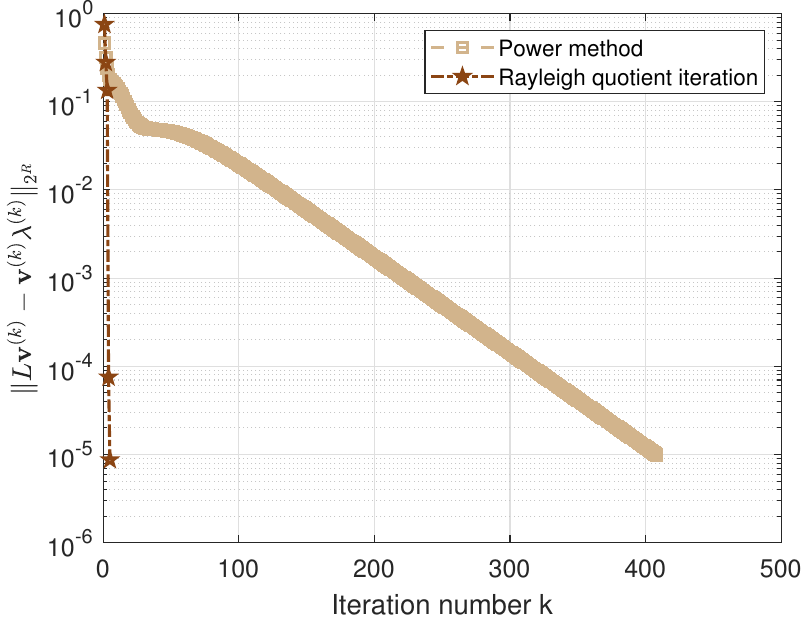}}
	\\ 
	\centering
	\subfigure[50-point circle, $L\in \mathbb{DQ}^{50\times 50}$]{\includegraphics[scale=0.57]{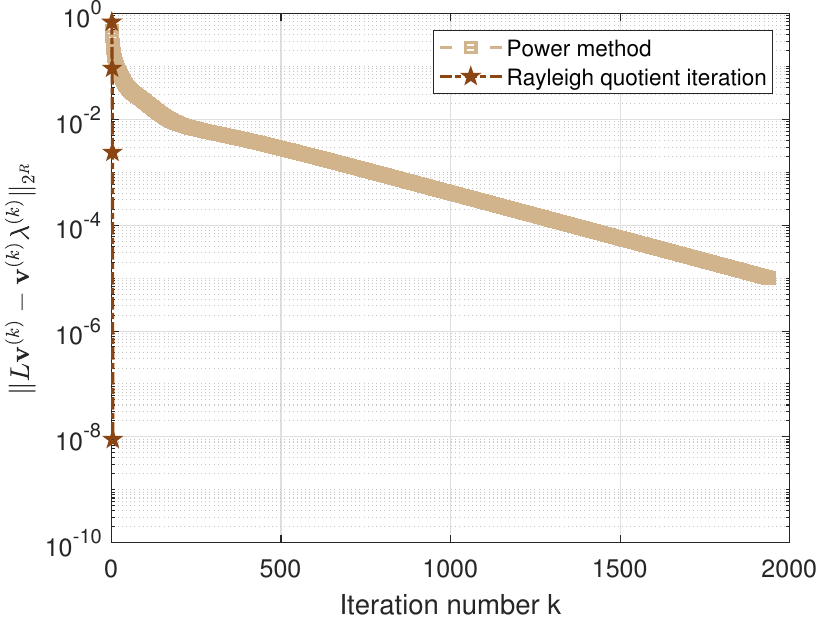}}
	\subfigure[100-point circle, $L\in \mathbb{DQ}^{100\times 100}$]{\includegraphics[scale=0.57]{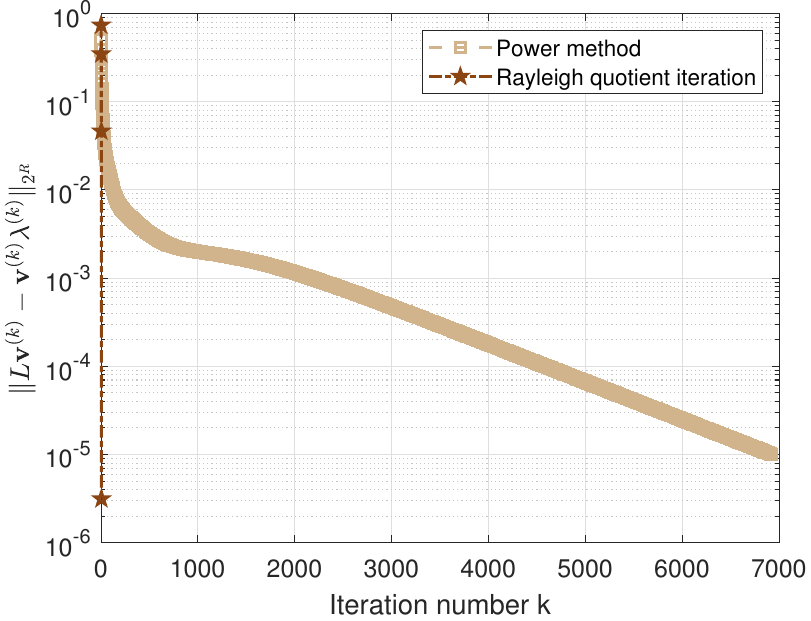}}
	\caption{The convergence curve of RQI and PM for computing an eigenvalue with the random Laplacian matrices.}
	\label{fig.3}
\end{figure}

Table \ref{tab2} lists their eigenvalues, the number of iterations, and CPU time of the Rayleigh quotient iteration and the power method \cite{cui2023power}. The symbol ``-" means that the iteration step $k$ has reached the upper limit of 15000, but it did not derive a solution.

As indicated by the data presented in Table \ref{tab2} and illustrated in Figure \ref{fig.3}, the Rayleigh quotient iteration demonstrates superior efficiency compared to the power method. It achieves rapid convergence in just a few iterations when calculating eigenvalues for Laplacian matrices of various sizes. This efficiency leads to shorter computation times and enhanced accuracy. These findings underscore the suitability of the Rayleigh quotient iteration method for computing eigenvalues in large scale dual quaternion Hermitian matrices.

\section{Conclusion}\label{section6}

\noindent In this paper, we have investigated the problem of solving right eigenvalues for dual quaternion Hermitian matrices. Firstly, we have presented the dual representation of dual quaternion matrices, which plays an important role in solving the dual quaternion linear systems. Secondly, leveraging this dual representation, we have proposed the Rayleigh quotient iteration for computing the eigenvalue with associated eigenvector of dual quaternion Hermitian matrices. Thirdly, by utilizing minimal residual property of the Rayleigh quotient, we have derived a convergence analysis for the Rayleigh quotient iteration. Finally, we have provided numerical examples to illustrate the efficiency of the Rayleigh quotient iteration for solving the dual quaternion Hermitian eigenvalue problem.

\vskip 3mm

\noindent \textbf{Funding}~ This work is supported by the National Natural Science Foundation of China under Grants 12371023, 12201149, 12361079, and the Natural Science Foundation of Guangxi Province under Grant 2023GXNSFAA026067.

\section*{Declarations}

\noindent \textbf{Conflict of interest}~ The authors declare no competing interests.

\end{document}